\documentclass[12pt]{article}

\pdfoutput=1

\usepackage{amsmath,amsthm,amssymb}
\usepackage[utf8x]{inputenc}
\usepackage{easybmat}
\usepackage{graphicx,hyperref}
\newtheorem{theorem}{Theorem}[section]
\newtheorem{lemma}[theorem]{Lemma}
\newtheorem{proposition}[theorem]{Proposition}

\theoremstyle{definition}
\newtheorem{definition}[theorem]{Definition}
\newtheorem*{remark}{Remark}

\usepackage[flushleft,alwaysadjust]{paralist}
\newcommand{\nin}{\noindent}
\newcommand{\vs}{\vspace{0.15in}}

\newcommand{\V}{{\bf V}}
\newcommand{\I}{{\bf I}}
\def \beq{ \begin{equation} }
\def \eeq{\end{equation}}

\title{Relative Equilibria in the Four-Vortex Problem \\ with Two Pairs of Equal Vorticities}
\author{Marshall Hampton\thanks{
Dept. of Mathematics and Statistics, University of Minnesota Duluth, mhampton@d.umn.edu},
Gareth E. Roberts\thanks{
Dept. of Mathematics and Comp. Sci., College of the Holy Cross, groberts@radius.holycross.edu}, 
Manuele Santoprete\thanks{
Dept. of Mathematics, Wilfrid Laurier Univerisity, msantoprete@wlu.ca}
}

\begin{document}

\maketitle

\begin{abstract}
We examine in detail the relative equilibria in the four-vortex problem where 
two pairs of vortices have equal strength,
that is, $\Gamma_1 = \Gamma_2 = 1$ and $\Gamma_3 = \Gamma_4 = m$ where $m \in \mathbb{R} - \{0\}$ is a
parameter.  One main result is that for $m > 0$, the convex configurations all contain a line of symmetry, 
forming a rhombus or an isosceles trapezoid.   The rhombus solutions exist for all $m$ but the isosceles trapezoid case exists only 
when $m$ is positive.  In fact, there exist asymmetric convex configurations when $m < 0$.  
In contrast to the Newtonian four-body problem
with two equal pairs of masses, where the symmetry of all convex central configurations is unproven, 
the equations in the vortex case are easier to handle, allowing for a complete classification of all solutions.
Precise counts on the number and type of solutions (equivalence classes) for different values of $m$, as
well as a description of some of the bifurcations that occur, are provided.
Our techniques involve a combination of analysis and modern and computational algebraic geometry. 
\end{abstract}


\tableofcontents

\section{Introduction}

The motion of point vortices in the plane is an old problem in fluid mechanics 
that was first given a Hamiltonian formulation by Kirchhoff \cite{kirchhoff_vorlesungen_1883}. 
This widely used model provides finite-dimensional approximations
to vorticity evolution in fluid dynamics.  The goal is to track the motion
of the point vortices rather than focus on their internal structure and deformation, 
a concept analogous to the use of ``point masses'' in celestial mechanics.
As with the Newtonian $n$-body problem, an important class of homographic solutions exist
where the configuration is self-similar throughout the motion.  
Such solutions are described as {\em stationary} by O'Neil~\cite{oneil_stationary_1987} and are not limited
to just relative equilibria (rigid rotations), but also include
equilibria, rigidly translating solutions and collapse configurations.

In this paper we focus primarily on the relative equilibria of the four-vortex problem when
two pairs of vortices have the same vortex strength.  Specifically,
if $\Gamma_i \in \mathbb{R} - \{ 0 \}$ denotes the vortex strength of the 
$i$-th vortex, we set $\Gamma_1 = \Gamma_2 = 1$ and $\Gamma_3 = \Gamma_4 = m$,
treating $m$ as a real parameter.  Our main goal is to classify and describe all solutions
as $m$ varies.  Four-vortex configurations which are not collinear (nor contain any three vortices which 
are collinear) can be classified as either {\itshape concave} or {\itshape convex}.   
A concave configuration has one vortex which is located strictly inside the convex hull of the other three, whereas
a convex configuration does not have a vortex contained in the convex hull of the other three vortices.  

The symmetry and asymmetry of solutions plays a major role in our analysis.
In fact, part of the motivation behind our work was to determine whether symmetry could be
proven for this choice of vorticities when $m > 0$ and the configuration is assumed to be 
convex.  This question, while  solved in the Newtonian four-body problem when the equal masses are assumed to be opposite each
other in a convex central configuration~\cite{albouy_symmetry_2008, perez-chavela_convex_2007},  is still open for the case when equal masses are 
assumed to be adjacent.  In contrast, in this paper, we are able to prove that symmetry is required in the case of four vortices with two equal pairs of vorticities.  
In particular, we show that any convex relative equilibrium with $m > 0$, and any concave
solution with $m < 0$, must have a line of symmetry.  For the convex case, the symmetric solutions are
a rhombus and an isosceles trapezoid.  In the concave case, the symmetric solution is an
isosceles triangle with an interior vortex on the axis of symmetry.

\begin{table}
\begin{centering}

\begin{tabular}{|c|c|c|}
\hline 
 Shape  &  $m \in (-1,1]$  & Type of solution (number of)\\
\hline
        &       &                                          \\
Convex   &  $m=1$ &  Square (6) \\[0.05in]
                 & $0 < m < 1$  & Rhombus (2), Isosceles Trapezoid (4) \\[0.05in]
                  & $-1 < m < 0$ & Rhombus (4)\\
         &       &  Asymmetric (8) \\[0.05in]                                                              
                 & $-1/2 < m < 0$ &  $\mbox{Kite}_{34}$ (4) \\[0.05in]
                 & $m^\ast < m < -1/2$ & $\mbox{Kite}_{12}$ (4) \\
         &       &                                          \\
Concave  & $m=1$  &  Equilateral Triangle with Interior Vortex (8) \\[0.05in]
         & $0 < m < 1$ &     $\mbox{Kite}_{34}$ (8)\\
         &                       &    Asymmetric (8)             \\ [0.05in]
         &$-1/2 < m<0$  &     $\mbox{Kite}_{12}$ (4)\\
         &       &                                           \\
Collinear  & $m=1$              &  Symmetric  (12)    \\ [0.05in]
                  &  $0 < m < 1$     & Symmetric  (4)         \\
                  &                            & Asymmetric  (8)        \\[0.05in]
                  & $-1 < m < 0$    & Symmetric  (2)           \\[0.05in]
                  & $-1/2 < m < 0$ & Asymmetric (4)         \\       
\hline 
\end{tabular} 

\end{centering}

\label{table:count}
\caption{The number of relative equilibria equivalence classes for the four-vortex problem
with vortex strengths $\Gamma_1 = \Gamma_2 = 1$ and $\Gamma_3 = \Gamma_4 = m$, in terms
of $m$ and the type of configuration. The special value $m^\ast \approx -0.5951$ is the only real root
of the cubic $9m^3 + 3m^2 + 7m + 5$.  $\mbox{Kite}_{ij}$ refers to a kite configuration with vortices
$i$ and~$j$ on the axis of symmetry.}
\end{table}

A precise count on the number and type of solutions as a function of the parameter $m$ is given in
Table~\ref{table:count}.  A configuration is called a {\em kite} if two vortices are on an axis of symmetry
and the other two vortices are symmetrically located with respect to this axis.  Kite configurations may
either be concave or convex.  When counting solutions, we use the standard convention from celestial mechanics
that solutions which are identical under scaling or rotation are considered equivalent.  Note that two
solutions identical under a reflection are counted separately.   The full set of solutions for $m = 2/5$, $m=-1/5$ and $m=-7/10$ (excluding 
any strictly planar configurations identical under a reflection) are shown in Figure~\ref{fig:all-sols}.

 \begin{figure}[h!]
 \centering
 \includegraphics[width=14cm,keepaspectratio=true]{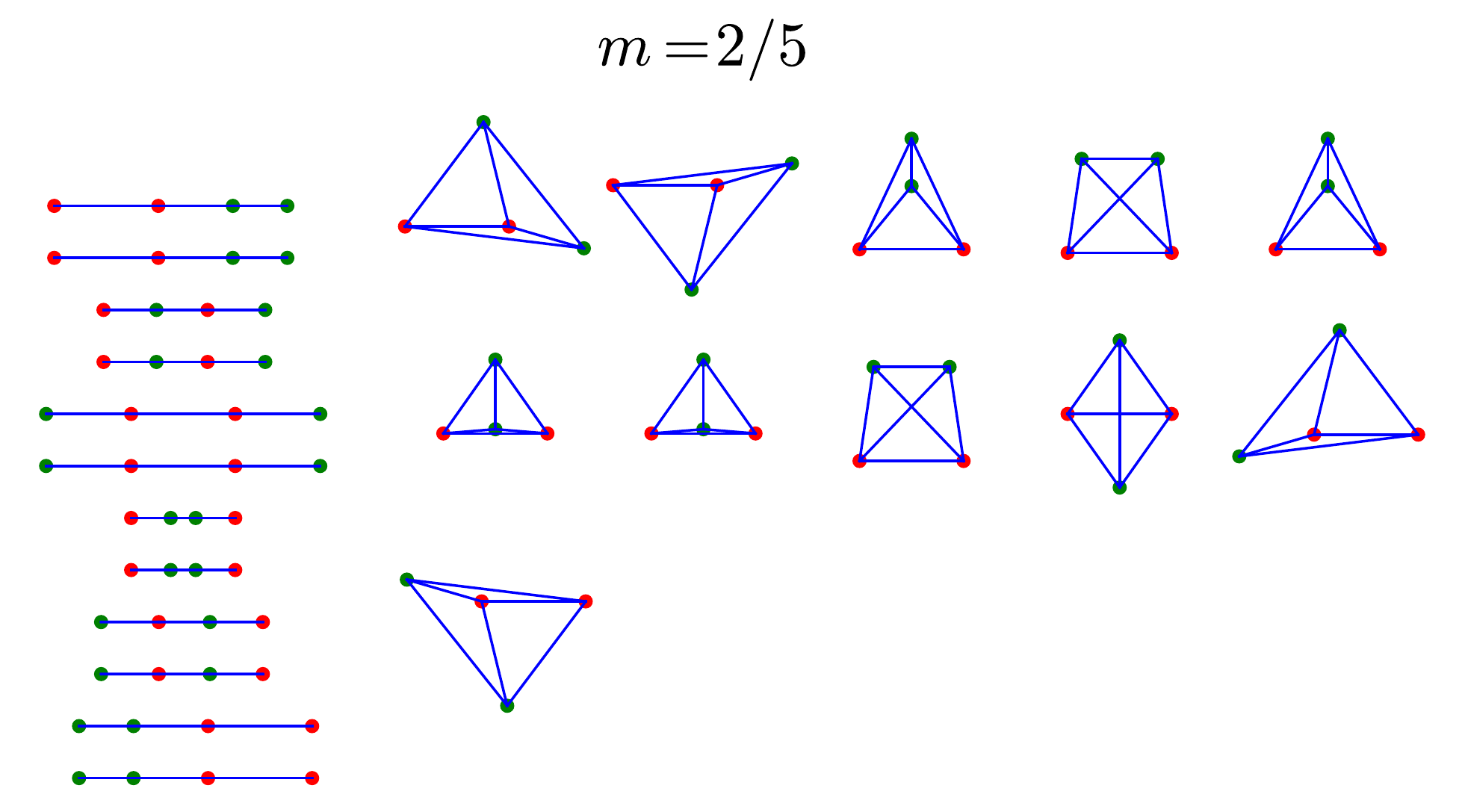}
 \includegraphics[width=12cm,keepaspectratio=true]{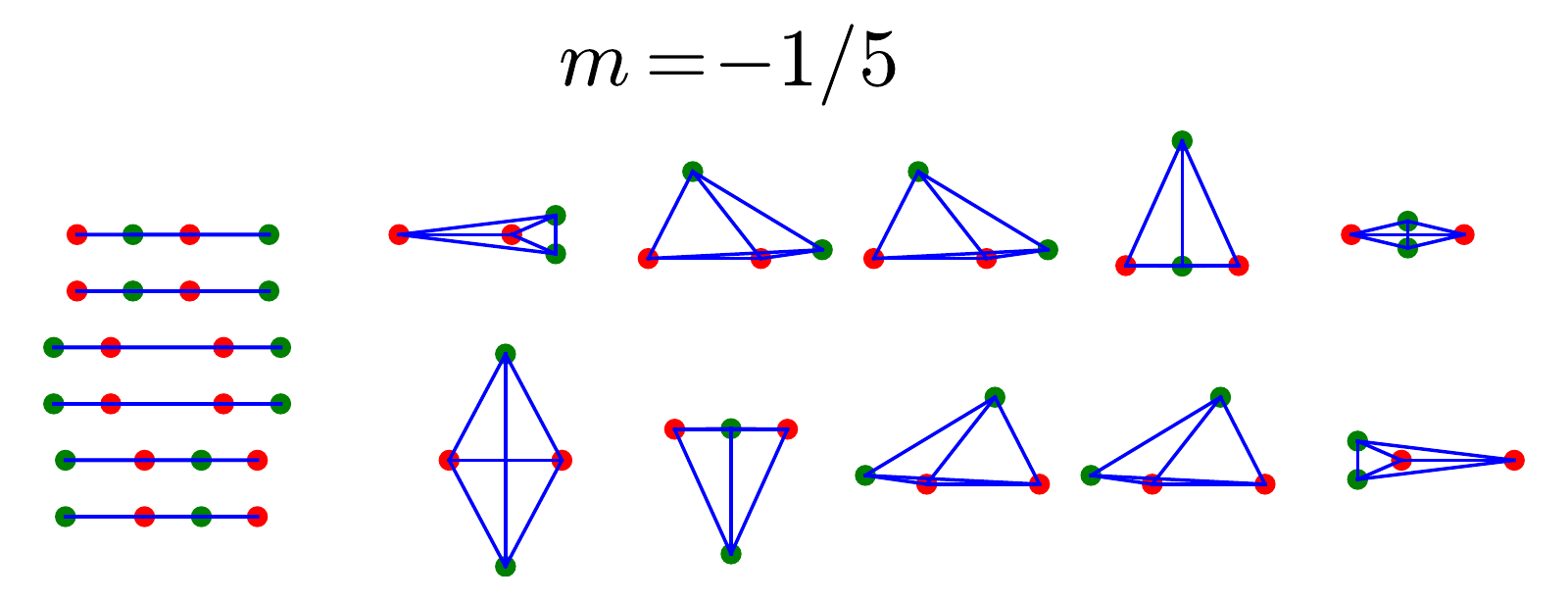}
 \vspace{0.2in}
 \includegraphics[width=12cm,keepaspectratio=true]{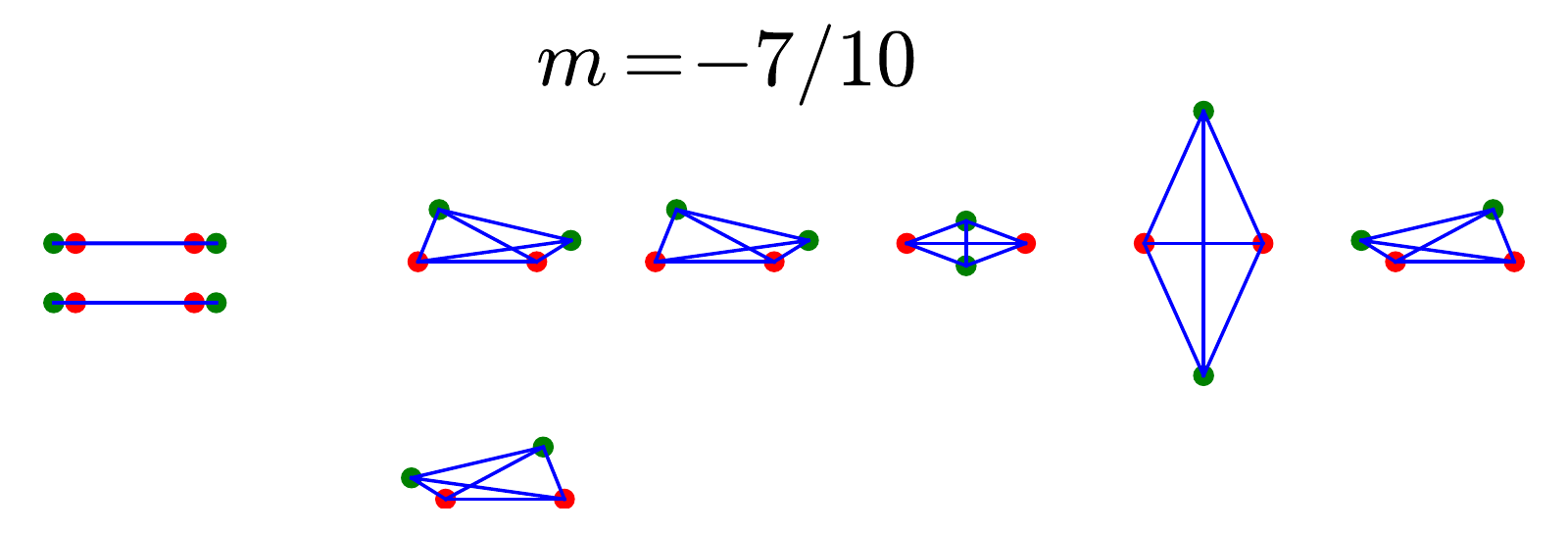}
 \caption{The full set of solutions for $m = 2/5$, $m = -1/5$ and $m = -7/10$. 
 Vortices $\Gamma_1 = \Gamma_2 = 1$ are denoted by red disks and vortices
 $\Gamma_3 = \Gamma_4 = m$ by green ones. }
 \label{fig:all-sols} 
 \end{figure}

When $m=1$ (all vortex strengths equal), there are 26 solutions, all symmetric.  Allowing for relabeling, there are only three geometrically
distinct configurations: a square, an equilateral triangle with a vortex at the center, and a collinear configuration.
This is different than the Newtonian case, where an additional symmetric, concave solution exists, consisting
of an isosceles triangle with an interior body on the axis of symmetry~\cite{albouy_symmetric_1996}. 
Part of the reason for the contrast is that the equilateral triangle with central vortex is degenerate when all
vortices have the same strength~\cite{meyer_bifurcations_1988, palmore_relative_1982}. 
An interesting bifurcation occurs as $m$ decreases through $m=1$, 
where the equilateral triangle solution splits into four different solutions.  
If vortex 3 or~4 is at the center of the triangle, then the solution 
for $m=1$ bifurcates into two different isosceles triangles with the interior vortex on the line of symmetry.  
In this case, the larger-strength vortices are on the base of the triangle.  If vortex 1 or~2 
is at the center of the triangle, the solution branches into two asymmetric concave configurations
that are identical under a reflection.  In this case, the weaker-strength
vortices are on the base of the triangle.  Thus, the number of solutions increases from 26 to 34
for the case $0 < m < 1$.  This concurs with Palmore's result in~\cite{palmore_relative_1982}, which
specifies a lower bound of 34 non-degenerate solutions for the four-vortex problem with positive vorticities.

As $m$ flips sign, there are two interesting bifurcation values at $m = -1/2$ and
$m = m^\ast \approx -0.5951$, the only real root of the cubic $9m^3 + 3m^2 + 7m + 5$.  
There are 26 solutions for $-1/2 < m < 0$, including a family of convex kite configurations
having the negative strength vortices on an axis of symmetry and the positive strength vortices
symmetrically located with respect to this axis.  Based on the primary result in~\cite{albouy_symmetry_2008},
which is applicable to our problem but only when $m > 0$, one might expect this convex kite to have two
axes of symmetry, forming a rhombus.   However, as demonstrated in Section~\ref{section:concave-kites},
this is not the case.  There is also a family of symmetric concave kite configurations having one of the positive
strength vortices in the interior of an isosceles triangle formed by the outer three.  
As $m$ approaches $-1/2$, the triangle containing the interior vortex and the base of the outer isosceles triangle
limits on an equilateral triangle while the fourth vortex (at the apex of the outer isosceles triangle) heads off to infinity.
As $m$ decreases through $-1/2$, this family bifurcates into a convex configuration, while the previous convex kite family disappears.
This new family of convex kites, with vortices 1 and 2 on the axis of symmetry, exists only for
$m^\ast < m < -1/2$.  In sum, there are 18 solutions for $m^\ast < m < -1/2$, and
14 solutions when $m = -1/2$ or when $-1 < m \leq m^\ast$.  The singular bifurcation at $m = -1/2$ occurs in part because 
the leading coefficient of a defining polynomial vanishes.  This is likely a
consequence of the fact that the sum of three vorticities vanishes when $m=-1/2$, a particularly troubling case when attempting to prove 
finiteness for the number of relative equilibria in the full four-vortex problem (see~\cite{hampton_finiteness_2009}).
The bifurcation at $m = m^\ast$ is a pitchfork bifurcation, as discussed in Section~\ref{rhombus_sec}.

In the next section we define a relative equilibrium and explain how to use mutual distances as variables in the four-vortex
problem.  In Section~\ref{section:algtechs}, we describe the relevant algebraic techniques used to analyze and quantify
the number of solutions.   Section~\ref{Section:SpecialCases} examines the interplay between symmetry and 
equality of vorticities in two special cases: equilibria and vanishing total vorticity.  Sections \ref{section:collinear},
\ref{section:asymmetric} and~\ref{section:symmetric} cover the collinear, asymmetric and symmetric cases, respectively, in considerable detail.
Throughout our research, symbolic computations (e.g., calculation of Gr\"obner bases) were performed using 
Sage~\cite{stein_sage_2011}, {\sc Singular}~\cite{decker_singular_2011}  and Maple${15}^\mathrm{TM}$.

\section{Relative Equilibria}

We begin with the equations of motion and the definition of a relative equilibrium for the 
$n$-vortex problem.  Let $J=\begin{bmatrix}0&1\\-1&0\end{bmatrix}$,
and let  $\nabla_j$ denote the two-dimensional partial gradient with respect to $x_j$.
A system of $n$ planar point vortices with vortex strengths $\Gamma_i\neq 0$ and positions $x_i\in \mathbb{R}^2$ evolves according to 
\begin{equation}\label{eqmotion}
\Gamma_i\dot x_i=J\nabla_i H = -J \sideset{}{'}\sum_{j=1}^n \frac{\Gamma_i\Gamma_j}{r_{ij}^2}(x_i-x_j), \quad 1\leq i\leq n
\end{equation}
where  $H=-\sum_{i<j}\Gamma_i\Gamma_j\log r_{ij}$, $r_{ij}=\|x_i-x_j\|$, and the prime on the summation indicates omission of the term with $j=i$. 

A {\em relative equilibrium} motion is  a solution of (\ref{eqmotion}) of the form $x_i(t)=c+e^{-J\lambda t}(x_i(0)-c)$, that is, a periodic 
solution given by a uniform rotation with angular velocity $\lambda\neq 0$ around some point $c\in \mathbb{R}^2$. Such a solution is possible if and only if 
the initial positions $x_i(0)$ satisfy the equations
\begin{equation}\label{eqcc}
- \lambda(x_i(0)-c)= \frac{1}{\Gamma_i}\nabla_i H =  \sideset{}{'}\sum_{j=1}^n \frac{\Gamma_j}{r_{ij}^2}(x_j(0)-x_i(0))
\end{equation}
for each $i \in \{1, \ldots, n \}$.
Denote $\Gamma = \sum_i\Gamma_i$ as the {\em total circulation} and assume for the moment that $\Gamma \neq 0$.
Multiplying the $i$-th equation in~(\ref{eqcc}) by $\Gamma_i$ and summing over $i$ shows that the center of rotation $c$ is equivalent
to the {\itshape center of vorticity}, 
$c=\frac{1} {\Gamma}\sum_i\Gamma_i x_i$.  If $\Gamma = 0$, then we obtain instead that the {\itshape moment of vorticity} 
$\sum_i \Gamma_i x_i$ must vanish.  

\begin{definition}
A set of initial positions $x_i(0)$ satisfying equation~(\ref{eqcc}) for each $i \in \{1, \ldots, n \}$ is called a {\itshape central configuration}. 
The corresponding rigid rotation with angular velocity $\lambda \neq 0$ is called a {\itshape relative equilibrium}.  
We will often use these two terms interchangeably. 
\end{definition}

Define the {\em moment of inertia} $I$ with respect to the center of vorticity as
\[
I=\frac 1 2\sum_{i=1}^n \Gamma_i\|x_i - c\|^2  \, .
\]
$I$ measures the size of the system.  We can then rewrite equation~(\ref{eqcc}) as 
\begin{equation}\label{eqcc2}
\nabla H + \lambda \nabla (I-I_0) = 0
\end{equation}
where $\nabla=(\nabla_1,\ldots, \nabla_n)$ and $I=I_0$.
Therefore, $\lambda$ can be viewed as a Lagrange multiplier and any solution of~(\ref{eqcc2}) can be interpreted as 
a critical point of the Hamiltonian $H(x)$ under the condition that $I$ remains constant.  
Using the homogeneity of the functions $H$ and $I$, equation~(\ref{eqcc2}) implies that the angular
velocity $\lambda$ in a relative equilibrium is given by
\begin{equation}
\lambda = \frac{L}{2I}  \quad \mbox{ where } \quad  L = \sum_{i < j}^n \Gamma_i \Gamma_j
\label{eq:lambda}
\end{equation}
is the {\em total vortex angular momentum}.  This implies that in the case where all vortex strengths are positive, $\lambda > 0$
and the relative equilibrium is rotating in the counterclockwise direction.  It is also important to note that for
any relative equilibrium with a particular set of vortex strengths $\Gamma_i$, we can scale the vorticities
by some common factor $\nu$ and maintain the relative equilibrium, but with a new angular velocity
$\nu \lambda$.  If $\nu < 0$, then the sign of $\lambda$ flips, as does the direction of rotation.

Unlike the $n$-body problem, it is possible to have equilibria in the $n$-vortex problem (where $\lambda = 0$).
For $n=4$, explicit solutions can be derived (see~\cite{hampton_finiteness_2009}).  
An analysis of the equilibrium solutions in the context of symmetry is presented in Section~\ref{symmetric_equilibria}.

\subsection{Using mutual distances as coordinates} 
\label{sec:distancesascoordinates}

We now consider the case of $n=4$ vortices. Our presentation  follows the approach of \cite{schmidt_central_2002} in 
describing the work of Dziobek~\cite{dziobek_ueber_1900} for the Newtonian $n-$body problem. We want to express equation~(\ref{eqcc2}) in 
terms of the mutual distance variables $r_{ij}$. Between four vortices  there are six mutual distances, which are not independent if the vortices are planar.  
In the planar case the distances satisfy the following condition, which can be interpreted as the vanishing of the volume of the 
tetrahedron formed by the four vortices:
\[e_{CM}=\begin{vmatrix} 0&1&1&1&1\\1&0&r_{12}^2&r_{13}^2&r_{14}^2\\1&r_{12}^2&0&r_{23}^2&r_{24}^2\\1&r_{13}^2&r_{23}^2&0&r_{34}^2\\ 
1&r_{14}^2&r_{24}^2&r_{34}^2&0 \end{vmatrix}=0.\]
The matrix in the above determinant is known as the Cayley-Menger matrix.

Hence, planar central configurations are obtained as critical points of 
\begin{equation} \label{lagrange_mult_eq}
H + \lambda (I-I_0)+ \frac{\mu}{32}e_{CM}.
\end{equation}
Using the homogeneity of $H, I$ and $e_{CM}$, the value of $\lambda$ in this setup is identical
to one given in equation~(\ref{eq:lambda}).
To find $\nabla e_{CM}$ restricted to planar configurations, we use the following important formula
\[
\frac{\partial e_{CM}}{\partial r_{ij}^2}=-32 A_iA_j
\]
where $A_i$ is the oriented area of the triangle $T_i$ whose vertices are all the vortices except for the $i$-th body. 
Setting the gradient of (\ref{lagrange_mult_eq}) equal to zero yields the equations
\[
\frac{\partial H}{\partial r_{ij}^2} + \lambda \frac{\partial I}{\partial r_{ij}^2} + \frac{\mu}{32} \frac{\partial e_{CM}}{\partial r_{ij}^2}=0.
\]

If $\Gamma \neq 0$, then $I$ can be written in terms of the mutual distances as
$$
I = \frac{1}{2 \Gamma} \sum_{i < j} \Gamma_i \Gamma_j r_{ij}^2 \, .
$$
Using this, we obtain the following equations for a four-vortex central configuration:
\begin{equation}
\Gamma_i \Gamma_j (r_{ij}^{-2}+\lambda') = \sigma A_i A_j
\label{eq:DzioStart}
\end{equation}
where $\lambda' = - \lambda/ \Gamma$, $\sigma = -2 \mu$, $I = I_0$ and $e_{CM} = 0$. 
If $\Gamma = 0$, then a different approach is more useful (see~\cite{celli_2005, hampton_finiteness_2009} for expositions of the equations).  
We discuss the role of specific symmetries when $\Gamma = 0$ in Section~\ref{gamma0_sec}.

Assuming $\Gamma \neq 0$, we group the equations in~(\ref{eq:DzioStart}) so that when they are multiplied together pairwise, their right-hand sides are identical:
\begin{equation}\label{eqDziobek1}\begin{split}
&\Gamma_1\Gamma_2(r_{12}^{-2}+\lambda')=\sigma A_1A_2, \quad \Gamma_3\Gamma_4(r_{34}^{-2}+\lambda')=\sigma A_3A_4\\
&\Gamma_1\Gamma_3(r_{13}^{-2}+\lambda')=\sigma A_1A_3, \quad \Gamma_2\Gamma_4(r_{24}^{-2}+\lambda')=\sigma A_2A_4\\
&\Gamma_1\Gamma_4(r_{14}^{-2}+\lambda')=\sigma A_1A_4, \quad \Gamma_2\Gamma_3(r_{23}^{-2}+\lambda')=\sigma A_2A_3.
\end{split}
\end{equation}
 This yields the well-known Dziobek equations~\cite{dziobek_ueber_1900}, but for vortices: 
\begin{equation}
(r_{12}^{-2}+\lambda')(r_{34}^{-2}+\lambda')=(r_{13}^{-2}+\lambda')(r_{24}^{-2}+\lambda')=(r_{14}^{-2}+\lambda')(r_{23}^{-2}+\lambda').
\label{eq:Dzio}
\end{equation}

From the different ratios of two masses that can be found from the equations in~(\ref{eqDziobek1}) we obtain the following equations:
\begin{equation}\label{eqDziobek2}
\begin{split}
 \frac{\Gamma_1A_2}{\Gamma_2A_1}&=\frac{\rho_{23}+\lambda'}{\rho_{13}+\lambda'}=\frac{\rho_{24}+\lambda'}{\rho_{14}+\lambda'}=\frac{\rho_{23}-\rho_{24}}{\rho_{13}-\rho_{14}}\\
\frac{\Gamma_1A_3}{\Gamma_3A_1}&=\frac{\rho_{23}+\lambda'}{\rho_{12}+\lambda'}=\frac{\rho_{34}+\lambda'}{\rho_{14}+\lambda'}=\frac{\rho_{23}-\rho_{34}}{\rho_{12}-\rho_{14}}\\
\frac{\Gamma_1A_4}{\Gamma_4A_1}&=\frac{\rho_{24}+\lambda'}{\rho_{12}+\lambda'}=\frac{\rho_{34}+\lambda'}{\rho_{13}+\lambda'}=\frac{\rho_{24}-\rho_{34}}{\rho_{12}-\rho_{13}}\\
\frac{\Gamma_2A_3}{\Gamma_3A_2}&=\frac{\rho_{13}+\lambda'}{\rho_{12}+\lambda'}=\frac{\rho_{34}+\lambda'}{\rho_{24}+\lambda'}=\frac{\rho_{13}-\rho_{34}}{\rho_{12}-\rho_{24}}\\
\frac{\Gamma_2A_4}{\Gamma_4A_2}&=\frac{\rho_{14}+\lambda'}{\rho_{12}+\lambda'}=\frac{\rho_{34}+\lambda'}{\rho_{23}+\lambda'}=\frac{\rho_{14}-\rho_{34}}{\rho_{12}-\rho_{23}}\\
\frac{\Gamma_3A_4}{\Gamma_4A_3}&=\frac{\rho_{14}+\lambda'}{\rho_{13}+\lambda'}=\frac{\rho_{24}+\lambda'}{\rho_{23}+\lambda'}=\frac{\rho_{14}-\rho_{24}}{\rho_{13}-\rho_{23}}\\
\end{split}
\end{equation}
where $\rho_{ij}=r_{ij}^{-2}$.

Eliminating $\lambda'$ from equation~(\ref{eq:Dzio}) and factoring yields the important relation
\begin{equation}
(r_{13}^2 - r_{12}^2)(r_{23}^2 - r_{34}^2)(r_{24}^2-r_{14}^2) \; = \; (r_{12}^2 - r_{14}^2)(r_{24}^2 - r_{34}^2)(r_{13}^2-r_{23}^2).
\label{eq:CCfactor}
\end{equation}
Assuming the six mutual distances determine an actual configuration in the plane, this equation is necessary and sufficient for the existence of a four-vortex relative equilibrium.  The corresponding vortex strengths are then found from the equations in~(\ref{eqDziobek2}).

Relationships between the lengths of the sides in a four-vortex central configuration follow from the equations in~(\ref{eqDziobek2})
and the signs of the oriented areas $A_i$.  If the configuration of vortices is concave, precisely three of the oriented areas have the same sign. 
In the convex case, two oriented areas are positive and two are negative.  

The sign of $\lambda' = - \lambda/ \Gamma$ for a given relative equilibrium is of some interest.  Note that scaling the vorticities by any $\nu \in \mathbb{R} - \{0\}$
does {\em not} change the value of $\lambda'$.  If all the vorticities have the same sign, then $\lambda' < 0$ is assured.  However, if the vorticities have different
signs, it is possible that $\lambda'$ could become positive.   

When $\lambda' > 0$, the equations in~(\ref{eqDziobek2}) imply that $\frac{\Gamma_i A_j}{\Gamma_j A_i} > 0$
for any choice of indices $i$ and $j$.  Taking $\Gamma_1 = \Gamma_2 = 1$ and $\Gamma_3 = \Gamma_4 = m$ with $m < 0$, we see that the only possible
solutions have $A_1, A_2 > 0$ and $A_3, A_4 < 0$, or $A_1, A_2 < 0$ and $A_3, A_4 > 0$.  The configuration must be convex with vortices 1 and~2 on one
diagonal and vortices 3 and~4 on the other.  We show in Section~\ref{sec:kites-lampos} that the configuration must have at least one axis of symmetry.
There exists a family of convex kite configurations for $m^\ast < m < -1/2$ and a family of rhombi (see Section~\ref{rhombus_sec}) for 
$-1 \leq m < -2 + \sqrt{3}$.  These are the only possible solutions to our problem having $\lambda' > 0$.

\begin{proposition}\label{prop:concave}
Suppose we have a concave central configuration with four vortices and $\Gamma \neq 0$.
 \begin{enumerate}
 \item If all the vorticities are positive (negative), then  all exterior sides are longer 
than the interior ones and $1/\sqrt{-\lambda'}$ is less than the lengths of all exterior sides
and greater than the lengths of all interior sides. 
 \item If two of the vorticities are positive and two are negative, then  the exterior sides 
 connecting vortices with vorticities of opposite sign and the interior side connecting vortices 
 with vorticities of the same sign have lengths greater than $1/\sqrt{-\lambda'}$, 
 while the remaining sides have length less than $1/\sqrt{-\lambda'}$.
 \end{enumerate}
\end{proposition}
\begin{proof}
 (1) Let $\Gamma_4$ be the interior vorticity and suppose that $A_4<0$.  Then we have $A_1, A_2, A_3 > 0$. 
 Assume that $\rho_{34} > -\lambda'$.  Then, 
 using the equations in~(\ref{eqDziobek2}), we obtain
 \[
 \rho_{12},\rho_{13},\rho_{23}<-\lambda'<\rho_{14},\rho_{24},\rho_{34},
 \]
or 
\[
r_{12},r_{13},r_{23}>1/\sqrt{-\lambda'}>r_{14},r_{24},r_{34},
\]
namely all the exterior edges are longer than the interior ones. 
On the other hand, if we had assumed $\rho_{34}<-\lambda'$, then the inequalities above would be reversed
and the configuration could not be realized geometrically since the interior sides cannot all be longer than the exterior sides. 

(2)  Let $\Gamma_4$ be the interior vorticity. Let $\Gamma_3, \Gamma_4<0$, and $\Gamma_1,\Gamma_2>0$. 
Furthermore, suppose that $A_4<0$.   Then we have $A_1, A_2, A_3 > 0$. 
If $\rho_{34}<-\lambda'$, then, using the equations in~(\ref{eqDziobek2}), we obtain 
 \[
 \rho_{13},\rho_{23},\rho_{34}<-\lambda'<\rho_{12},\rho_{14},\rho_{24},
 \]
or 
\[
r_{13},r_{23},r_{34}>1/\sqrt{-\lambda'}>r_{12},r_{14},r_{24}.
\]
If we had  assumed $\rho_{34}>-\lambda'$, then the inequalities above would be reversed and the configuration
could not be realized geometrically. We show this with a proof by contradiction. Assume  $\rho_{34}>-\lambda'$. 
Then $|A_4|<|A_3|$ since the two triangles have one side in common
and $r_{24}, r_{14} > r_{23}, r_{13}$. But this is absurd since $T_4$ contains $T_3$.

 Let $\Gamma_1$ be the interior vorticity. Let $\Gamma_3,\Gamma_4<0$, and $\Gamma_1,\Gamma_2>0$. 
Furthermore, suppose that $A_1<0$.  Then we have $A_2, A_3, A_4 > 0$. 
If $\rho_{34}>-\lambda'$ then, using the equations in~(\ref{eqDziobek2}), we obtain 
 \[
 \rho_{12},\rho_{23},\rho_{24}<-\lambda'<\rho_{13},\rho_{14},\rho_{34},
 \]
or 
\[
r_{12},r_{23},r_{24}>1/\sqrt{-\lambda'}>r_{13},r_{14},r_{34}.
\]
If we had assumed $\rho_{34}<-\lambda'$ then the inequalities above would be reversed, and the configurations 
could not be realized geometrically (the proof is similar to the argument in the previous paragraph).
\end{proof}

Before we state an analogous theorem for the convex case we recall the following useful geometric lemma: 

\begin{lemma}\label{geometriclemma}
The combined length of the diagonals of a convex quadrilateral is greater than the combined length of any pair of its opposite sides. 
\end{lemma}
\begin{proof}
 Order the vortices counterclockwise. Let $o$ be the position vector describing the intersection of the diagonals, and let $r_{io}=\|x_i-o\|$.
 Applying the triangle inequality to the triangle of vertices $x_1,x_2$ and $o$ and to the one of vertices $x_3,x_4$ and $o$ yields
 \[
 r_{1o}+r_{2o}>r_{12} \quad \mbox{ and } \quad  r_{3o}+r_{4o}>r_{34}.
 \]
 Adding these two inequalities together, we obtain $r_{13}+r_{24} > r_{12}+r_{34}$. 
 A similar reasoning can be applied to the remaining two triangles. 
\end{proof}

\begin{proposition}\label{prop:convex}
Suppose we have a convex central configuration with four vortices, with $\Gamma \neq 0$ and $\lambda' < 0$.
 \begin{enumerate}
  \item If all the vorticities are positive (negative), then all exterior sides are shorter than the diagonals. 
Furthermore, the lengths of all exterior sides are less than $1/\sqrt{-\lambda'}$ and the lengths of all 
the diagonals are greater than  $1/\sqrt{-\lambda'}$. The shortest and longest exterior sides have to face each other. 
  \item If two of the vorticities are positive and adjacent, and two remaining ones are negative, then the exterior 
sides connecting vortices with vorticities of opposite sign have length less than $1/\sqrt{-\lambda'}$.  All the other sides have length greater than $1/\sqrt{-\lambda'}$.
  \item If two of the vorticities are positive and opposite, and two are negative, then either all the sides have length  
less than  $1/\sqrt{-\lambda'}$ , or all the sides have length greater than $1/\sqrt{-\lambda'}$.
 \end{enumerate}
\end{proposition}
\begin{proof}
(1) The proof is analogous to the one for the Newtonian four-body problem (see~\cite{schmidt_central_2002}).

(2) Order the vortices counterclockwise and let $\Gamma_1, \Gamma_2>0$ and $\Gamma_3,\Gamma_4<0$. The sign of the 
four areas are then $A_1>0, A_3>0$ and $A_2<0, A_4<0$. If we assume that $\rho_{34}>-\lambda'$, then
we find from the equations in~(\ref{eqDziobek2}) that 
\[
\rho_{12},\rho_{13},\rho_{24},\rho_{34}<-\lambda'<\rho_{14},\rho_{23},
\]
or
\[
r_{12},r_{13},r_{24},r_{34}>1/\sqrt{-\lambda'}>r_{14},r_{23}.
\]
On the other hand, if we had assumed that $\rho_{34}<-\lambda'$, then the inequalities above would be reversed and, 
by Lemma~\ref{geometriclemma}, the configuration could not be realized geometrically.

(3) Let $\Gamma_1,\Gamma_2>0$ and $\Gamma_3,\Gamma_4<0$, and assume the vortices of the same sign are opposite
one another. 
The sign of the four areas are then $A_1 > 0, A_2 > 0$ and $A_3 < 0, A_4 < 0$. If we assume that $\rho_{34} > -\lambda'$, 
then we find from the equations in~(\ref{eqDziobek2}) that 
\[
\rho_{12},\rho_{13},\rho_{24},\rho_{34},\rho_{14},\rho_{23}>-\lambda',
\]
or
\[
r_{12}, r_{13}, r_{24}, r_{34}, r_{14}, r_{23} < 1/\sqrt{-\lambda'}.
\]
If  $\rho_{34} < -\lambda'$, then the inequalities are reversed.
\end{proof}

\subsection{Symmetric configurations}
\label{section:SymmConf}

One immediate consequence of the Dziobek equations~(\ref{eq:Dzio}) and of equation~(\ref{eq:CCfactor}) is that if two mutual distances
containing a common vortex are equal (e.g., $r_{12} = r_{13}$), then the same equality of distances is true for the 
excluded vortex (here, $r_{24} = r_{34}$).  Specifically, if $i, j, k, l$ are distinct indices, then we have
\begin{equation}
r_{ij} = r_{ik} \; \mbox{ if and only if } \; r_{lj} = r_{lk}.
\label{eq:symm1}
\end{equation}
This relation is independent of the vortex strengths although $\Gamma_j = \Gamma_k$ necessarily must follow.
Any configuration satisfying equation~(\ref{eq:symm1}) has an axis of symmetry containing vortices $i$ and~$l$,
forming a kite configuration.  It may be either convex or concave, but it cannot contain three vortices
on a common line due to the equations in~(\ref{eqDziobek1}).  Unlike the Newtonian four-body problem, since there is no restriction here on the signs of the 
vortex strengths, it follows that {\em any} kite configuration has a corresponding set of vorticities 
that make it a relative equilibrium.  Assuming the configuration is not an equilateral triangle
with a vortex at the center, the vorticities are unique up to a common scaling factor and can be
determined by the equations in~(\ref{eqDziobek2}).

Next we consider the case where two mutual distances without a common index are equal.
In other words, suppose one of the following three equations holds:
\begin{equation}
r_{12} = r_{34},  \qquad  r_{13} = r_{24},  \qquad r_{14} = r_{23}.
\label{eq:symm2}
\end{equation}
Then, it does not necessarily follow that another pair of mutual distances must be equal.
However, if different pairs of vortices are assumed to be of equal strength, then we can conclude
an additional symmetry.  This fact will be important in Section~\ref{sec:elim-symm} 
when verifying that symmetry is required in certain cases.

\begin{lemma}
{\bf (Symmetry Lemma)}  Suppose that we have a strictly planar four-vortex relative equilibrium
with $\Gamma_1 = \Gamma_2$, $\Gamma_3 = \Gamma_4$ and $\Gamma \neq 0$.
Then, 
\begin{equation}
r_{13} \; = \; r_{24}  \quad \mbox{ if and only if } \quad r_{14} \; = \; r_{23}.
\label{eq:MainSymm}
\end{equation}  
If either equation in~(\ref{eq:MainSymm}) holds,
the configuration is convex and has either one or two axes of symmetry.  
In this case, the configuration is either 
an isosceles trapezoid with vortices 1 and~2 on one base, and 3 and~4
on the other, or it is a rhombus with vortices 1 and~2 opposite each other. 
\label{lemma:symm}
\end{lemma}

\begin{proof}

\nin  Suppose that $r_{13} = r_{24}$.  Then, by one of the equations in~(\ref{eqDziobek1}), 
$\Gamma_1 \Gamma_3 (r_{13}^{-2} + \lambda') = \Gamma_2 \Gamma_4 (r_{24}^{-2} + \lambda')$
implies that 
\begin{equation}
A_1 A_3 \; = \; A_2 A_4.
\label{eq:ProdArea}
\end{equation}
We also have that 
\begin{equation}
A_1 + A_2 + A_3 + A_4  \; = \;  0
\label{eq:SumArea}
\end{equation}
since the $A_i$'s are oriented areas.  Solving equation~(\ref{eq:SumArea}) 
for $A_4$ and substituting into equation~(\ref{eq:ProdArea}) yields the relation
$$
(A_1 + A_2)(A_2 + A_3) \; = \; 0.
$$

There are two possibilities.  First, suppose that $A_1 = -A_2$.  Then equation~(\ref{eq:SumArea}) immediately
implies $A_4 = -A_3$.  
The configuration is convex due to the signs of the $A_i$'s and must have the side containing vortices
1 and~2 parallel to the side containing vortices 3 and~4.
Then, by another equation in~(\ref{eqDziobek1}), $\sigma A_1 A_4 = \sigma A_2 A_3$ implies that 
$$
\Gamma_1 \Gamma_4 (r_{14}^{-2} + \lambda') \; = \;  \Gamma_2 \Gamma_3 (r_{23}^{-2} + \lambda').
$$
Since $\Gamma_1 = \Gamma_2$ and $\Gamma_3 = \Gamma_4$, it follows that $r_{14} = r_{23}$.  The configuration must be
an isosceles trapezoid with congruent legs and diagonals, and base lengths given by $r_{12}$ and~$r_{34}$.

Next, suppose that $A_3 = -A_2$.  Then, $A_4 = -A_1$ immediately follows from equation~(\ref{eq:SumArea}).
The configuration is convex and must have the side containing vortices
1 and~4 parallel to the side containing vortices 2 and~3.  Since $r_{13} = r_{24}$, the configuration is
either an isosceles trapezoid with $r_{12} = r_{34}$ or is a parallelogram with $r_{14} = r_{23}$.
For the isosceles trapezoid, the lengths $r_{13}$ and $r_{24}$ can correspond to either a pair of congruent
diagonals or a pair of congruent legs, depending on the ordering of the vortices.  However, for the parallelogram,
these lengths must correspond to a pair of opposite sides as the diagonals do not have to be congruent.

It turns out that, in this case, the isosceles trapezoid is actually a square.  To see this, 
by an equation in~(\ref{eqDziobek1}),  $\sigma A_1 A_2 = \sigma A_3 A_4$ implies that  
$$
\Gamma_1^2 (r_{12}^{-2} + \lambda') \; = \;  \Gamma_3^2 (r_{34}^{-2} + \lambda').
$$
Thus, if $r_{12} = r_{34}$, then $\Gamma_1 = \Gamma_3$ and all vortex strengths are equal
(the case $\Gamma_1 = -\Gamma_3$ is excluded since $\Gamma \neq 0$).  In this case, it is
straight-forward to show that the isosceles trapezoid reduces to a square (see Section~\ref{isos-trap})
and thus $r_{14} = r_{23}$.

In the case of the parallelogram, we have $A_1 = A_2 = -A_3 = -A_4$.  By the
equations in~(\ref{eqDziobek1}), this implies that $r_{13} = r_{14} = r_{23} = r_{24}$ and the configuration
is a rhombus with vortices 1 and~2 opposite each other.  This proves the forward implication.

The proof in the reverse direction is similar.  If $r_{14} = r_{23}$, then we derive $A_1 A_4 = A_2 A_3$
from an equation in~(\ref{eqDziobek1}).  Taken with equation~(\ref{eq:SumArea}), this yields 
$$
(A_1 + A_2)(A_1 + A_3) \; = \; 0.
$$
As before, the case $A_1 = -A_2$ leads to an isosceles trapezoid with $r_{13} = r_{24}$ and base lengths
given by $r_{12}$ and $r_{34}$.  The case $A_1 = -A_3$ leads to either the square or a rhombus
configuration with vortices 1 and~2 across from each other.  In either configuration we deduce that $r_{13} = r_{24}$.
This completes the proof.
\end{proof}

\vs

\begin{remark}
\begin{enumerate}
\item  Similar results exist if different pairs of vortices are assumed to be equal.  For example, if 
$\Gamma_1 = \Gamma_3$ and $\Gamma_2 = \Gamma_4$, then $r_{12} = r_{34}$ if and only
if $r_{14} = r_{23}$.  

\item  The result is also valid in the Newtonian four-body problem (and for other potentials of the same
form) since it only depends on the geometry of the configuration
and the inherent structure of the equations in~(\ref{eqDziobek1}).
\end{enumerate}
\end{remark}

\subsection{The Albouy-Chenciner equations}

For the remainder of the paper (excluding the special cases discussed in Section~\ref{Section:SpecialCases}), 
we will assume the equality of vortex strengths specified in Lemma~\ref{lemma:symm}.
Specifically, we set $\Gamma_1 = \Gamma_2 = 1$ and $\Gamma_3 = \Gamma_4 = m$,
treating $m$ as a real parameter.  Without loss of generality, we restrict to the case where $m \in (-1,1]$.
The choice $m=-1$ implies that $\Gamma = 0$, a special case examined in Section~\ref{gamma0_sec}.

When $\Gamma \neq 0$, the equations for a relative equilibrium can be written in polynomial form as
\[
f_{ij}=\sum_{k=1}^n\Gamma_k[S_{ik}(r_{jk}^2-r_{ik}^2-r_{ij}^2)+S_{jk}(r_{ik}^2-r_{jk}^2-r_{ij}^2)]=0,
\]
where $1\leq i<j\leq n$ and the $S_{ij}$ are given by
\[
S_{ij}=\frac{1}{r_{ij}^2}+\lambda'\quad (i\neq j), \quad\quad S_{ii}=0.
\]
These very useful equations are due to Albouy and Chenciner~\cite{albouy_probleme_1997} 
(see also \cite{hampton_finiteness_2005} for a nice derivation).  They form a polynomial system in the $r_{ij}$
variables after clearing the denominators in the $S_{ij}$ terms. 

Since any relative equilibrium may be rescaled, we will impose the normalization
$\lambda' = -1$ unless otherwise stated. This usually can be assumed without loss of generality.   
However, as explained in Section~\ref{sec:distancesascoordinates}, for $m < 0$ the normalization $\lambda'=1$ also
needs to be considered.  This case is discussed in Sections \ref{sec:kites-lampos} and~\ref{rhombus_sec}.
We denote the complete set of polynomial equations determined by $f_{ij} = 0$ as $\mathcal{F}$.

From the Albouy-Chenciner equations we can derive a  more restrictive set of equations, namely
\[
g_{ij}=\sum_{k=1}^n\Gamma_kS_{ik}(r_{jk}^2-r_{ik}^2-r_{ij}^2)=0.
\]
Since $g_{ij}\neq g_{ji}$, these give 12 distinct equations. We denote the complete set of polynomial equations determined by
$g_{ij} = 0$ as $\mathcal{G}$, and we will refer to them as the unsymmetrized Albouy-Chenciner equations.

The solutions of the Albouy-Chenciner equation give configurations of all dimensions, but, in the four-vortex problem,  
they can be specialized to the strictly planar case by adding the three Dziobek equations in~(\ref{eq:Dzio}).
Introducing the variables $s_{ij}=r_{ij}^2$, these equations can be written as
\[
h_{ijkl}=(s_{ij}^{-1} + \lambda')(s_{kl}^{-1} + \lambda')-(s_{ik}^{-1} + \lambda')(s_{jl}^{-1} + \lambda') = 0,
\]
where $i,j,k$ and $l$ are all distinct indeces. We denote the set of Dziobek equations (with denominators
cleared) as $\mathcal{H}$.

\section{Algebraic Techniques}
\label{section:algtechs}

In this section we briefly describe three of our main algebraic techniques for analyzing solutions to our problem:
elimination theory using Gr\"{o}bner bases, a useful lemma to distinguish when the roots of a quartic 
are real or complex, and Mobius transformations.

\subsection{Gr\"obner bases and elimination theory}

We mention briefly some elements from elimination theory and the theory of Gr\"obner bases that will prove useful
in our analysis.  For a more detailed exposition see~\cite{cox_ideals_2007}.

Let $K$ be a field and consider the polynomial ring $K[x_1,\ldots,x_n]$ of polynomials in $n$ variables over $K$. 
Let $f_1, \ldots f_l$ be $l$ polynomials in $K[x_1,\ldots,x_n]$ and consider the ideal $I=\langle f_1,\ldots, f_l\rangle$ generated by these polynomials.
Denote ${\bf V}(I)$ as the affine variety of $I$.

\begin{definition}
 An admissible order $>$ on $K[x_1,\ldots,x_n]$ is  called a {\em $k-$elimination order} if 
 \[
 x_1^{a_1}\ldots x_n^{a_n}>x_{k+1}^{b_{k+1}}\ldots x_n^{b_n}
 \]
 when $a_{i_0}>0$ for some $i_0\in\{1,\ldots,k\}$.
\end{definition}

A lexicographical (lex) order is an example of $k-$elimination order for all $k$.

\begin{definition}
 The {\em $k$-th elimination ideal} $I_k$ is the ideal of $K[x_{k+1},\ldots, x_n]$ defined by
 \[
 I_k=I\cap K[x_{k+1},\ldots,x_n]
 \]
\end{definition}

Gr\"obner bases provide a systematic way of finding elements of $I_k$ using the proper term ordering.

\begin{theorem}{\bf (The Elimination Theorem)}
 Let $I$ be an ideal of $K[x_1,\ldots,x_n]$ and let $G$ be a Gr\"obner basis of $I$ with respect to a $k-$elimination order for $k$ where $0\leq k\leq n$. Then the set 
 \[
 G_k=G\cap K[x_{k+1},\ldots,x_n]
 \]
 is a Gr\"obner basis of the $k$-th elimination ideal $I_k$.
\end{theorem}

Gr\"obner bases also provide a method of determining when an element of 
 ${\bf V}(I_k)$ (a partial solution) can be extended to a full solution in ${\bf V}(I)$.
 This can be achieved by repeatedly applying the following theorem.

\begin{theorem}{\bf (The Extension Theorem)}
Let $K$ be an algebraically closed field and let $I$ be some ideal in $K[x_1, \ldots, x_n]$.
Let $G_{k-1}$ be a lex Gr\"obner basis for the elimination ideal $I_{k-1}$ and write each
polynomial in $G_{k-1}$ as
$$
g_i = h_i(x_{k+1}, \ldots x_n) x_{k}^{N_i} + \mbox{terms where $x_{k}$ has degree $< N_i$},
$$
where $N_i > 0$ and $h_i$ is nonzero.
Suppose that $(a_{k+1}, \ldots, a_n)$ is a partial solution in ${\bf V}(I_k)$ and that
$h_i(a_{k+1}, \ldots, a_n) \neq 0$ for some index $i$.  Then there exists $a_k \in K$
such that $(a_k, a_{k+1}, \ldots, a_n) \in {\bf V}(I_{k-1})$.
\end{theorem}

\subsection{A useful lemma for quartic polynomials}

The analysis of the collinear case, as well as some strictly planar cases, frequently involves solving a quartic
equation whose coefficients are polynomials in~$m$.   We state here some useful results about quartic equations.
Consider the general quartic polynomial $\zeta(x) =  ax^4+bx^3+cx^2+dx+e$, with coefficients
in $\mathbb{R}$ and $a \neq 0$.
We first remove the cubic term of $\zeta$ by the change of variables $x=y-\frac{b}{4a}$.  This produces the polynomial
$a \, \eta(y)$, where $\eta$ is the {\em shifted quartic} $\eta(y) = y^4+py^2+qy+r$.  The {\em discriminant} $\Delta$ of any polynomial is a positive constant times
the square of the product of all possible differences of roots.  The discriminant of $\zeta$ is equivalent to
$a^6$ times the discriminant of $\eta$.
For a general quartic, it is straight-forward to check that if
$\Delta > 0$, then the roots are either all real or all complex (two pairs of complex conjugates).
If $\Delta < 0$, then there are two real roots and two complex roots.
The roots are repeated if and only if $\Delta = 0$.

Let $y_1, y_2, y_3$ and $y_4$ be the four roots of $\eta(y)$.  By construction, 
$y_1 + y_2 + y_3 + y_4 = 0$.  It follows that  
$$
z_1 = -(y_1+y_2)^2, \quad z_2 = -(y_1+y_3)^2, \quad z_3 = -(y_1+y_4)^2
$$
are the roots of the {\em resolvent cubic} $\xi(z) = z^3 - 2pz^2 + (p^2-4r) z + q^2$.
The discriminant of~$\xi$ is equivalent to the discriminant of~$\eta$.
When the discriminant is positive, the resolvent cubic is particularly useful for determining whether the roots
of $\eta(y)$ are all real or all complex.  Specifically, if the four roots of $\eta(y)$ are real,
then $z_1, z_2$, and $z_3$  must be negative and real (or if $q=0$, then $z_i = 0$
for precisely one $i$ while the remaining $z_i$'s are negative).  On the other hand,
if the four roots of $\eta(y)$ are complex, then one of the roots of $\xi(z)$ is less than
or equal to zero, but two of the roots of $\xi(z)$ must be positive and real.  
These facts can easily be translated into conditions on the
coefficients of the resolvent cubic.

\begin{lemma}
Suppose that the discriminant of a quartic polynomial $\zeta(x)$ is positive and
let $\eta(y) = y^4+py^2+qy+r$ be the shifted quartic related to $\zeta$.
Then, the four distinct roots of $\zeta$ are real if and only if
$p < 0$ and $p^2 - 4r > 0$.
\label{lemma:quartic-roots}
\end{lemma}

\begin{proof}
Since the discriminant is positive, the roots of $\zeta$ are distinct and either all real
or all complex.  First suppose that the four roots of $\zeta(x)$ are real.  Since $\eta(y)$ is obtained through
a simple translation, it follows that the four roots of $\eta(y)$ are distinct and real.  
Then, the three real roots $z_1, z_2, z_3$ of the resolvent cubic $\xi(z)$ are negative with possibly
one root equal to zero (if $q=0$).  This in turn
implies that 
\begin{eqnarray*}
z_1 + z_2 + z_3 &=&  2p \; < \; 0 \quad \mbox{ and } \\
z_1 z_2 + z_1 z_3 + z_2 z_3 & = & p^2 - 4r  \; > \; 0.
\end{eqnarray*}

In the other direction, suppose that $p < 0$ and $p^2 - 4r > 0$.  This implies that all the coefficients
of $\xi(z)$ are positive, and by Descartes' rule of signs, $\xi(z)$ has no positive, real roots.  It follows
that the four roots of $\eta(y)$ cannot be complex.  This shows that the roots of $\zeta(x)$ are
real.
\end{proof}

\subsection{M\"obius transformations}

M\"obius transformations are a key tool for isolating
roots of polynomials since they allow for a dramatic reduction on the number of variations
of signs in the coefficients of a polynomial~\cite{Barros_set_2011}.
Let $\mathcal{P}(x_1,\ldots,x_n)$  be  a multivariate polynomial in $n$ variables.  
We use M\"obius transformations of the form  
\[
x_i=\frac{k_i^{(2)}y_i+k_i^{(1)}}{y_i+1}
\]
where $k_i^{(1)}<k_i^{(2)}$ and $i=1,\ldots n$, to obtain changes of variables for  $\mathcal{P}(x_1,\ldots,x_n)$.
The numerator of the rational function thus obtained is a multivariate  polynomial  in the variables $y_1,\ldots,y_n$ 
that restricted to the set $[0,\infty)\times\ldots\times[0,\infty)$ has the same number of roots as $\mathcal{P}(x_1,\ldots,x_n)$
 restricted to $[k_1^{(1)},k_1^{(2)})\times\ldots\times [k_n^{(1)},k_n^{(2)})$. 

 M\"obius transformations can be used to determine partitions of the
 space into blocks so that, in each block, the coefficients of the polynomials
under consideration have a very simple behavior with respect to variations
of signs. Therefore, we can use such transformations, together with Descartes' rule of signs, to determine whether 
a polynomial has a root or does not have a root in a given region of space.

\section{Special Cases}
\label{Section:SpecialCases}

In this section, we use some results from~\cite{hampton_finiteness_2009} to examine two special cases:
equilibrium solutions and vanishing total vorticity ($\Gamma = 0$).

\subsection{Symmetric equilibria} \label{symmetric_equilibria}

Equilibria are solutions to equation~(\ref{eqcc}) with $\lambda = 0$.
If we choose coordinates so that the fourth vortex is at the origin, and the third vortex is at $(1,0)$, 
then the other two vortices in a four-vortex equilibrium must be located at:
\begin{equation}\label{equilibria_locs}
\begin{split}
x_1 = \frac{1}{2 (\Gamma_2 + \Gamma_3 + \Gamma_4)} \left (2\Gamma_4 + \Gamma_2 , \pm \sqrt{3}\, \Gamma_2 \right )\\
x_2 = \frac{1}{2 (\Gamma_1 + \Gamma_3 + \Gamma_4)} \left (2\Gamma_4 + \Gamma_1 , \mp \sqrt{3}\, \Gamma_1 \right )
\end{split}
\end{equation}
where the sign chosen for $x_2$ is the opposite of that for $x_1$ (see~\cite{hampton_finiteness_2009} for details).

According to equation~(\ref{eq:lambda}), for an equilibrium to exist the vorticities must satisfy $L = 0$. 
This is impossible if all the vorticities are equal.  If three vorticities are equal, for example with $\Gamma_1 = \Gamma_2 = \Gamma_3 = m$, then $\Gamma_4 = -m$.  
We can rescale so that $m=1$ without loss of generality.  Using~(\ref{equilibria_locs}) it is easy to see that the only equilibria are equilateral triangles with the opposing vortex located in the center.

Now consider equilibria when $\Gamma_1 = \Gamma_2 = 1$ and $\Gamma_3 = \Gamma_4 = m$.  Then,
$L=0$ implies $m= -2 \pm \sqrt{3}$.  Since we are restricting to $m \in (-1,1]$, we have $m = -2 + \sqrt{3} \approx -0.2679$.  
In this case, the equilibria are rhombi with vortices 1 and 2 opposite each other, and $\displaystyle \frac{r_{34}}{r_{12}} = 2 - \sqrt{3} = -m$.  
These rhombi are members of one of the families of rhombi described in Section~\ref{rhombus_sec}.

Suppose we instead specify the symmetry of the configuration.  From~(\ref{equilibria_locs}) it is immediate that there are no 
collinear equilibria for nonzero vorticities.  There are also no isosceles trapezoid equilibria.  To see this, 
consider a trapezoid with $r_{14} = r_{23}$ and $r_{13} = r_{24}$.  Then~(\ref{equilibria_locs}) implies that
$$
\frac{\Gamma_2}{\Gamma_2 + \Gamma_3 + \Gamma_4} = \frac{-\Gamma_1}{\Gamma_1 + \Gamma_3 + \Gamma_4},
$$
but there are no real nonzero vorticities satisfying this equation and $L=0$.

Finally, there is the case in which the configuration is a kite (either concave or convex).  We choose vortices 1 and 2 to be on an axis of symmetry, 
so $r_{13} = r_{14}$ and $r_{23} = r_{24}$. Then equation~(\ref{equilibria_locs}) implies that $\Gamma_3 = \Gamma_4$.
If we choose any $\Gamma_2 \neq -2\Gamma_4$ with
$$
\Gamma_2 = -\frac{2 \Gamma_{1} \Gamma_{4} + \Gamma_{4}^{2}}{\Gamma_{1} + 2 \Gamma_{4}},
$$
then $L=0$ is satisfied and there is a kite equilibrium given by~(\ref{equilibria_locs}).  If we fix $\Gamma_4 = 1$, 
then as $\Gamma_2 \rightarrow -2$, $\Gamma_1 \rightarrow \infty$ and the 
configuration of vortices 2, 3, and 4 approaches an equilateral triangle while the ordinate of vortex~1 heads off to $\pm \infty$.

\subsection{Vanishing total vorticity}\label{gamma0_sec}

When the total vorticity $\Gamma = 0$, the analysis of stationary vortex configurations (i.e., those that do not change their shape) 
requires equations adapted to this special case.  For four vorticities we simply apply some of the results from~\cite{hampton_finiteness_2009}.

If three of the vorticities are equal and $\Gamma = 0$, we can consider without loss of generality the case 
$\Gamma_1 = \Gamma_2 = \Gamma_3 = 1$ and $\Gamma_4 = -3$.  It is interesting that there are six asymmetric 
configurations which rigidly translate (three pairs of configurations, within a pair the configurations are reflections of one another), 
but no symmetric rigidly translating solutions.  These configurations can be obtained directly from the equations in \cite{hampton_finiteness_2009}.  

For relative equilibria we use the equations
$$
S_1 = S_2 = S_3 = S_4 = s_0
$$
and
\begin{equation}\label{eq_celli}
\frac{1}{s_{12}}+\frac{1}{s_{34}} = \frac{1}{s_{13}}+\frac{1}{s_{24}} = \frac{1}{s_{14}}+\frac{1}{s_{23}},
\end{equation}
where 
$$
S_i = \Gamma_j s_{ij} + \Gamma_k s_{ik} +  \Gamma_l s_{il}, \ \ \ \ \{i,j,k,l\} = \{1,2,3,4\}
$$
and $s_0$ is an auxiliary variable.  Recall that $s_{ij} = r_{ij}^2$.  We clear denominators in the equations from (\ref{eq_celli}) to get a polynomial system.  
Using a Gr\"{o}bner basis to eliminate $s_0$, we find that for $\Gamma_1 = \Gamma_2 = \Gamma_3 = 1$ and $\Gamma_4 = -3$ 
there are two types of symmetric relative equilibria.  The first is the equilateral triangle with vortex 4 at its center.  
The second type is a concave kite with the three equal vorticities on the exterior isosceles triangle.  If we scale the 
exterior triangle so that its longest side is length $1$, then the base is length $\sqrt{\sqrt{3} - 5}$, and the other sides of the 
interior triangle containing vortex $4$ are length $\sqrt{-1 + \sqrt{3}}$.

If two pairs of vorticities are equal and $\Gamma = 0$, then we have $\Gamma_1 = \Gamma_2 = 1$ and $\Gamma_3 = \Gamma_4 = -1$.  
In this case there are no rigidly translating solutions.  This is somewhat surprising since each pair of opposing vortices would rigidly translate if 
unperturbed.  Using the same equations as we did for the case where three vorticities are equal, we find that there are two relative equilibria, 
each of which forms a rhombus.  These appear in the two families of rhombi described in Section~\ref{rhombus_sec}, 
with $\displaystyle \frac{s_{34}}{s_{12}} = 3 \pm 2\sqrt{2}$.

\section{Collinear Relative Equilibria}
\label{section:collinear}

Collinear relative equilibria of the four-vortex problem can be studied directly from equation~(\ref{eqcc}) since in this case it reduces to 
\begin{equation}
- \lambda (x_i - c) \; = \;  \sum_{j \neq i}^n  \frac{\Gamma_j}{x_j -x _i}  \quad \mbox{ for each } i \in \{1, 2, 3, 4\},
\label{eq:coll-orig}
\end{equation}
where $c, x_i \in \mathbb{R}$ rather than $\mathbb{R}^2$.  
Clearing denominators from these equations yields a polynomial system.  
Rather than fix $\lambda$ or $c$, we use the homogeneity and translation invariance of the system and set $x_3 = -1$ and $x_4 = 1$.  
As specified earlier, we set $\Gamma_1 = \Gamma_2 = 1$ and $\Gamma_3 = \Gamma_4 = m \in (-1,1]$, treating
$m$ as a parameter.

\subsection{Symmetric solutions}

\label{sec:coll-symm}

Given our setup, symmetric configurations correspond to solutions where either $x_1 = - x_2$
or $r_{12} = 2$.   Both cases are simple enough to analyze using Gr\"{o}bner bases.

\subsubsection*{Case 1: $x_1 = - x_2$}

In this case the center of vorticity $c$ is at the origin.  The solutions are the roots of an even
quartic polynomial in $x_1$, given by 
$
m \, x_1^4 - 5(m + 1) x_1^2  + 1.
$
The discriminant of this quartic changes sign at $m=0$ and there are
four real roots for $m>0$ but only two when $m < 0$.
The result is that for any $m \in (-1,1]$, there is a collinear relative equilibrium with  
\begin{equation}
x_1 = -x_2 = \pm \sqrt{\frac{5 \, m + 5 - \sqrt{25 m^{2} + 46 m + 25} }{2 m}}
\label{eq:coll-symm1}
\end{equation}
(with $x_1 = -x_2 =  \pm 1/\sqrt{5}$ when  $m=0$).  This solution has 
vortices $x_1$ and $x_2$ symmetrically located between vortices $x_3$ and $x_4$,
and as $m \rightarrow -1^+$, the inner vortices approach the outer vortices
(with collision at $m=-1$.)
For $m \in (0,1]$, $|x_1|$ decreases monotonically in $m$ from
$1/\sqrt{5} \approx 0.4472$ to $\sqrt{3} - \sqrt{2} \approx 0.3178$, while 
for $m \in (-1,0)$, $|x_1|$ decreases monotonically in $m$ from
$1$ to $1/\sqrt{5}$.

There is an additional collinear relative equilibrium if $m \in (0,1]$ given by
$$
x_1 = -x_2 = \pm \sqrt{\frac{5 \, m + 5 + \sqrt{25 m^{2} + 46 m + 25} }{2 m}}.
$$
In this case, the vortices $x_1$ and $x_2$ are symmetrically located outside 
vortices $x_3$ and $x_4$, and their positions approach $\pm \infty$
as $m \rightarrow 0^+$.  The value of $|x_1|$ decreases monotonically in $m$ from
$\infty$ to $\sqrt{3} + \sqrt{2} \approx 3.1462$.

\subsubsection*{Case 2: $r_{12} = 2$}

In this case the center of vorticity $c$ is not necessarily located at the origin; however,
the configuration will be symmetric about some fixed point equidistant from both the inner and outer pairs
of vortices.  Since $|x_1 - x_2| = 2$ and $x_1 = -x_2$ together imply $|x_2| = 1$ (collision), the two cases 
are distinct.  If $m \neq 1$, a solution for case 2 will have the inner and outer pair of vortices having different circulations.
It turns out that this case is impossible when $m \neq 1$.  

To see this, we consider the polynomials obtained from system~(\ref{eq:coll-orig}) along with 
$x_1 - x_2 - 2$ and $u(x_1 + x_2) - 1$ (to eliminate solutions from Case~1).
Computing a lex Gr\"{o}bner basis for this set of polynomials   
quickly yields $m=1$.  The same result is obtained when using
$x_1 - x_2 + 2$.  Moreover, in each computation, a 
fourth-degree polynomial in $x_2$ is obtained
that provides the exact solution for the special case $m=1$.

In sum, there are no solutions with $r_{12} = 2$ unless $m=1$.
When $m=1$, there are 8 solutions with $r_{12}=2$ given by $(x_1,x_2) =$
\begin{equation}
(1 \pm \sqrt{2} \pm \sqrt{6}, -1 \pm \sqrt{2} \pm \sqrt{6})
\quad \mbox{and } \quad
(-1 \pm \sqrt{2} \pm \sqrt{6}, 1 \pm \sqrt{2} \pm \sqrt{6}),
\label{eq:coll-m=1}
\end{equation}
where the signs in a particular ordered pair are chosen to have the same pattern in
each coordinate (e.g., $(1 + \sqrt{2} - \sqrt{6}, -1 + \sqrt{2} - \sqrt{6})$).
It is straight-forward to check that each of these 8 solutions is a scaled
version of a solution found in Case 1.  Specifically, for each of the solutions
listed in~(\ref{eq:coll-m=1}) as well as the solutions obtained in Case~1 for $m=1$, 
there is a rescaling, translation and relabeling of the vortices that maps the solution onto 
\begin{equation}
x_3 = -1, \: x_2 = -\sqrt{3} + \sqrt{2}, \: x_1 = \sqrt{3} - \sqrt{2}, \: x_4 = 1,
\label{eq:coll-Hermite}
\end{equation}
a solution obtained from equation~(\ref{eq:coll-symm1}) when $m=1$.
We note that solution~(\ref{eq:coll-Hermite}) can be rescaled to coincide with the roots
of the Hermite polynomial $H_4$, as expected (see \cite{aref_point-vortex_2007}).

\subsection{Asymmetric solutions}

To locate any asymmetric solutions, we introduce the variables $u$ and $v$ along with the equations
$$
u(x_1 + x_2) - 1 \quad \mbox{ and } \quad 
v(x_1 - x_2)  - 1.
$$
Adding these two equations to the original polynomial system obtained
from~(\ref{eq:coll-orig}), we compute
a Gr\"{o}bner basis $G_{{\rm col}}$ with respect to the lex order
where $c > \lambda > u > v > x_1 > x_2 > m.$
This basis has 15 elements, the first of which is
an even, 8th-degree polynomial in $x_2$ with coefficients in $m$,
and a basis for the elimination ideal $I_5 = I \cap \mathbb{C}[x_2,m]$.
Introducing the variable $w = x_2^2$, this polynomial is given by
$$
\begin{array}{cl}
\zeta(w) = & m^2 (m+2) (1+2m)^2 \, w^4  - 4m (15m^4+61m^3+91m^2+61m+15) \, w^3 \\[0.05in]
& + (300m^5 + 1508m^4 + 2910m^3 + 2696m^2 + 1188m + 200) \, w^2 \\[0.05in]
& - 4(5m+4)(25m^4 + 127m^3 + 231m^2 + 175m + 45) \, w + (m + 2)^3 .
\end{array}
$$
If we eliminate $x_2$ instead of $x_1$, the same polynomial is obtained with
$w = x_1^2$.  The discriminant of $\zeta$ is 
$$
1048576 (m+2)^2 (m+1)^6 (1+2m)  (25m^2+58m+25)^3  m^2 (q_u(m))^2
$$
where $q_u(m) =  2m^5 - 16m^4 - 96m^3 - 162m^2 - 108m - 25$.  
For $m \in (0,1]$, the discriminant is strictly positive.  However,
on the interval $[-1,0]$, the discriminant vanishes at six different $m$-values.
In increasing order, these values are
$$
-1, m_0 \approx -0.6833, m_1 \approx -0.6066, m_2 \approx -0.5721, -1/2, 0.
$$ 
The values $m_0$ and $m_1$ are roots of the quintic polynomial $q_u$, and
$m_2 = (-29+6 \sqrt{6})/25$ is the largest root of the quadratic $25m^2+58m+25$.
For $m=-1$, $\zeta$ has a repeated root at $1$ of multiplicity four.
At $m = -1/2$, $\zeta$ becomes a cubic polynomial with roots $-3$ and $1$ (multiplicity 2).
Thus, the values $m=-1$ and $m=-1/2$ have no physical solutions (just collisions). 
For $m=0$, the quartic reduces to a quadratic, and 
we have two possible solutions at $w = (9 \pm 4 \sqrt{5})/5$ which correspond
to limiting configurations as $m \rightarrow 0^+$.  The other three values where
the discriminant vanishes also have no physical solutions, as we show below.

\begin{lemma}
The quartic polynomial $\zeta(w)$ has precisely four positive real roots for $m \in (0,1]$ and
two positive real roots for $m \in (-1/2,0)$.  For $m \in (-1,-1/2)$, there are no positive
roots except when $m = m_0$, where there is one positive, repeated root of multiplicity two.  
\label{lemma:coll-quartic}
\end{lemma}

\begin{proof}
Using Lemma~\ref{lemma:quartic-roots},  let $\eta$ be the shifted quartic of $\zeta(w)$.
From $\eta$, the key quantities $p$ and $p^2-4r$ are found to be
$$
p = - \frac{2(25m^2 + 58m + 25) (3m^4 + 14m^3 + 45m^2 + 54m + 19) (m+1)^2}
{m^2 (m+2)^2 (1+2m)^4}
$$
and
$$
p^2 - 4r  =  \frac{16(25m^2 + 58m + 25) (m+1)^2 t(m)}{m^4 (m+2)^4 (1+2m)^8}
$$
where
$$
\begin{array}{cl}
t(m) = & 75m^{12}+896m^{11}+5528m^{10}+23492m^9+77272m^8+197816m^7+376194m^6 \\[0.05in]
&  + 509968m^5 + 478976m^4 + 302388m^3 + 121964m^2 + 28320m + 2875.
\end{array}
$$

For $m \in (0,1]$, since the discriminant is positive, $p < 0$ and $p^2 - 4r > 0$, 
Lemma~\ref{lemma:quartic-roots} applies and the roots of $\zeta(w)$ are
all real.  The fact that they are all positive follows from Descartes' rule of signs since $\zeta(w)$ has
four sign changes while $\zeta(-w)$ has none.  For $m \in (-1/2,0)$, the discriminant remains positive
and using Sturm's theorem, one can show that $p < 0$ and $p^2 - 4r > 0$ both continue to hold.  It follows that
$\zeta(w)$ has four real roots in this case as well.  However, two of the roots are positive and two
are negative since both $\zeta(w)$ and $\zeta(-w)$ each have two sign changes.

For the case $m \in (m_2,-1/2)$, the discriminant is negative and thus there are precisely two real roots.
To show that the real roots are both negative, we note that at $m = m_2$, the quartic~$\zeta$
has a repeated root of multiplicity four at $w_2 = -(11+4 \sqrt{6} \,)/5 < 0$.  We then
can check that $\zeta(w_2)$ as a function of $m$ is strictly negative for $m \in (m_2,-1/2)$.
Since the leading coefficient of $\zeta$ is positive and since $\zeta(0) = (m+2)^3 > 0$,
it follows that both real roots must be negative in this case.  There are no positive
roots at $m = m_2$ since $w_2 < 0$ is the only root.

For the interval $m \in (-1,m_2)$, the discriminant is positive except when it vanishes at
$m_0$ and $m_1$.  However, either $p > 0$ or $p^2 - 4r < 0$ (or both) for $m$-values 
on this interval.  Consequently, by Lemma~\ref{lemma:quartic-roots}, if $m \in (-1,m_2) - \{m_0,m_1\}$, 
then $\zeta$ has four complex roots.

It remains to check the two cases $m=m_0$ and $m=m_1$, where the discriminant of
$\zeta$ vanishes.  Since the roots are complex in a neighborhood of each parameter
value, one has to check how the roots become repeated as the key parameter value
is approached.  It turns out that in each case, a pair of complex conjugate roots
meet on the real axis.  To see this, we use Descartes' rule of signs to 
check that the resolvent cubic $\xi$ has one negative root and a repeated positive
root at both $m=m_0$ and $m=m_1$.  It follows that one pair of complex conjugate roots
has merged into a real root since any other scenario would require that zero be a double root of the
resolvent cubic.  Because the quintic $q_u(m)$ is squared in the discriminant, the roots
return to being two pairs of complex conjugates after $m$ passes through these special 
parameter values.

To determine whether the repeated real root of $\zeta$ is positive or negative, we note that at either 
value $m=m_0$ or $m=m_1$, the shifted quartic $\eta(w)$ has the form
$$
(w - r_1)^2 [(w + r_1)^2 + r_2^2]  \; = \;
 w^4 + (r_2^2 - 2r_1^2) \, w^2 - 2 r_1 r_2^2 \, w + r_1^2 (r_1^2 + r_2^2).
$$
Thus, the coefficient of the linear term of $\eta$ can be used to determine the sign of the 
repeated real root.  Using Sturm's theorem, one can check that this coefficient is negative
at $m = m_0$ resulting in a positive value of $r_1$.  Then, since $\eta(w) = \zeta(w+\hat{c})$,
where $\hat{c} > 0$ at $m = m_0$, we know that the repeated real root of $\zeta$ when 
$m = m_0$ is positive.  To find this root exactly, we calculate the resultant of
$\zeta(w)$ and $\frac{d}{dw} (\zeta)$ with respect to $m$.  This produces a quintic in $w$ given by
$$
5w^5 - 53 w^4 + 98 w^3 + 198 w^2 + 9w - 1 .
$$
The largest root of this quintic, which is approximately $6.9632775$, is the positive, repeated root
of $\zeta(w)$ when $m = m_0$.  For the special case $m=m_1$, the derivative of $\zeta(w)$ is a cubic with only one real negative root.
Consequently, the repeated real root in this case is negative.
\end{proof}

While it is interesting to discover that the special parameter value $m = m_0$ has a unique solution,
it cannot lead to a physical solution to the problem because $x_1$ necessarily has to be a {\em different}
root of $\zeta$ (since we are excluding the symmetric solutions).  But the other roots of $\zeta$ 
are complex.  The fact that $x_1$ is complex is also confirmed by solving one of the 
polynomials in the Gr\"{o}bner basis $G_{{\rm col}}$ for $x_1$ 
when $m = m_0$ and $x_2 = \sqrt{6.9632775}$.

We summarize our findings for the collinear case in the following theorem.

\begin{theorem}
For the case $m \in (0,1]$, there are 12 collinear solutions, one for each possible
ordering of the vortices.  If $m=1$, all solutions are symmetric and geometrically equivalent to
the same configuration; otherwise, there are 4 symmetric solutions and 8 asymmetric solutions.
For $m \in (-1/2,0)$, there are a total of 6 solutions, 2 symmetric and 4 asymmetric,
while for $m \in (-1,-1/2]$, there are only 2 symmetric solutions and no asymmetric solutions.
\end{theorem}

\begin{proof}
The second polynomial in the Gr\"{o}bner basis $G_{{\rm col}}$, denoted $g_2$, is in 
the elimination ideal $I_4 = I \cap \mathbb{C}[x_1,x_2,m]$ and is linear in the variable $x_1$.
The coefficient of $x_1$ is found to be
$8(2m+1)(m+2)^2 \cdot q_u(m)$, which is nonzero for $m \in (-1/2,1]$.
By the Extension Theorem, we can extend any solution $(x_2,m)$ of $\zeta = 0$ to
a partial solution $(x_1,x_2,m)$ in ${\bf V}(I_4)$.  Moreover, since $g_2$ is linear in
$x_1$, there is at most one such solution, and $x_1$ must be real.  

Next, we check that the value of $x_1$ is distinct from the value of $x_2$ (excluding
collision) and $-x_2$ (excluding the symmetric solutions.)  This is accomplished by substituting
$x_1 = x_2$ into $g_2$ and then computing the resultant of this polynomial with $\zeta(x_2,m)$.
A polynomial in $m$ is thus obtained and it is easily checked that this polynomial has no roots
for $m \in (-1/2,1]$.  Therefore, the value of $x_1$ obtained by solving $g_2 = 0$ is distinct from
$x_2$.  A similar calculation, using the substitution $x_1 = -x_2$ in $g_2$, shows that
$x_1 \neq -x_2$ as well.

At this point, the Extension Theorem can be applied four more times using four of the basis polynomials
in $G_{{\rm col}}$ that are linear in the variables $v, u, \lambda$ and $c$, respectively, each
having nonzero leading coefficients.  Thus, fixing an $m \in (-1/2,1]$, 
for each positive root of $\zeta(w)$, we obtain two possible values of $x_2$, each of which 
extends uniquely to a full, real, asymmetric solution of our problem.  The precise count on the number of 
solutions for each case then follows directly from Lemma~\ref{lemma:coll-quartic} and Section~\ref{sec:coll-symm}.

The fact that there is one collinear relative equilibrium for each ordering of the vortices when 
$m > 0$ follows from a straight-forward generalization of a well-known result in the Newtonian $n$-body problem due to
Moulton~\cite{moulton_straight_1910}.  In brief, for each of the $n!$ connected components of the phase space for the
collinear $n$-vortex problem, there is a unique minimum of $H$ restricted to the ellipsoid $I = I_0$ 
(see p. 33 of~\cite{meyer_hall_offin_2009} or Section 6.1 of~\cite{oneil_stationary_1987} for details).
Each such minimum is a collinear relative equilibrium.  Identifying solutions
equivalent under a $180^\circ$ rotation of the plane gives a final count of $n!/2$.
\end{proof}

\begin{theorem}
When $m \in (-1,0)$, the signs of the vorticities in a collinear relative equilibrium must be arranged
as $+ - - +$ (symmetric case only), or as $+ - + -$ or $- + - +$
(when asymmetric solutions exist).  Therefore, when $m < 0$, it is not possible to have a collinear solution where both pairs
of vortices with the same strength are adjacent to each other.  
\end{theorem}

\begin{proof}
The symmetric case has already been analyzed at the start of this section.  For the asymmetric case,
we note that 
$$
\zeta(1) \; = \; -128 (1 + 2m) (m+2)^2 (m+1)^2 
$$
is strictly negative when $m \in (-1/2,0)$.  We also note that 
$\zeta(0) > 0$ and the leading coefficient of $\zeta(w)$ is always positive.
Since $\zeta$ has only two positive roots for $m \in (-1/2,0)$, it follows that
one root is less than one and the other is larger than one. (The two roots approach
one as $m \rightarrow -1/2^+$.)  The values of $x_1^2$ and $x_2^2$ must each be
roots of $\zeta(w)$.  Because we are only considering asymmetric solutions,
we have that $|x_1| < 1$ and $|x_2| > 1$ or vice versa.  Therefore, when $m \in (-1/2,0)$,
the only possible asymmetric orderings of vortices have signs arranged 
as $+ - + -$ or $- + - +$.
\end{proof}




\section{Asymmetric Strictly Planar Relative Equilibria}
\label{section:asymmetric}

\subsection{Eliminating symmetric solutions}
\label{sec:elim-symm}

In this section we study strictly planar relative equilibria that do not have a line of symmetry.
A major result proved here is that any convex solution with $m > 0$ or any
concave solution with $m < 0$ must contain a line of symmetry.
We prove this by saturating the Gr\"obner basis in order to eliminate any symmetric
solutions, and then showing the resulting system has no real solutions.

Let $\mathcal{\widetilde F}$ and  $\mathcal{\widetilde G}$ be the Albouy-Chenciner and the unsymmetrized 
Albouy-Chenciner equations in terms of $s_{ij}=r_{ij}^2$, respectively, with $\lambda' = -1, \Gamma_1 = \Gamma_2 = 1$, and 
$\Gamma_3=\Gamma_4=m \neq -1$. Let $\tilde e_{CM}$ be the Cayley-Menger determinant 
written in terms of $s_{ij}=r_{ij}^2$ and let $\mathcal{\widetilde H}$ be the Dziobek equations
with $\lambda' = -1$.  The polynomials 
$\mathcal{\widetilde F}, \mathcal{\widetilde G}, \tilde e_{CM}$, 
and $\mathcal{\widetilde H}$ belong to the polynomial ring 
$\mathbb{C}[m,s_{12},s_{13},s_{14},s_{23},s_{24},s_{34}]$. 

We begin by finding a Gr\"obner basis for the ideal  $I_s=\langle \mathcal{\widetilde F},\mathcal{\widetilde G},\tilde e_{CM}, \mathcal{\widetilde H} \rangle$. In order to accomplish this
we first find a Gr\"obner basis $G_{J_s}$ for the ideal $J_s=\langle \mathcal{\widetilde F},\mathcal{\widetilde G},\tilde e_{CM}\rangle$ and then we 
compute a Gr\"obner basis $G_{I_s}$ for $I_s = \langle G_{J_s}, \mathcal{\widetilde H} \rangle$.
At this stage we saturate with respect to the variables $s_{13},s_{14}$, and $s_{24}$, to eliminate possible solutions where one of the mutual distances has zero length. 
We also saturate with respect to $(s_{13}-s_{24})$, $(s_{14}-s_{23})$, $(s_{13}-s_{14})$, $(s_{23}-s_{24})$, $(s_{13}-s_{23})$, and $(s_{14}-s_{24})$.
Due to Lemma~\ref{lemma:symm} and equation~(\ref{eq:symm1}), saturating with respect to these differences is equivalent to eliminating any symmetric solutions. 
We denote the resulting Gr\"obner basis as $\widetilde G_{I_s}$.

Computing $\widetilde G_{I_s}$ with respect to an elimination order that eliminates all the variables except  $s_{12}$ and $s_{34}$
yields the following system of two equations in two unknowns:
\[
s_{34} m + s_{12} - m - 1=0, \quad \quad s_{12}^{2} - 2 s_{12} s_{34} + s_{34}^{2} - 1=0.
\]
This system has the two solutions
\begin{equation}
(s_{12},s_{34})=\left(\frac{2m+1}{m+1},\frac{m}{m+1}\right), \quad \quad (s_{12},s_{34})=\left(\frac{1}{m+1}, \frac{m+2}{m+1}\right),
\label{sols-s12s34}
\end{equation}
where one solution is sent to the other one by the transformation $m \rightarrow \frac{1}{m}$ and
$s_{12} \leftrightarrow s_{34}$.
If we compute the Gr\"obner basis of $\tilde G_{I_s}$ with respect to an elimination order that eliminates all but one 
$s_{ij} = x$, 
we obtain a polynomial that is the product  of the two polynomials 
\[\begin{split} p_1 = \; &4 \, {\left(m^2 + 2 \, m + 1\right)} x^{4} - 4 \, {\left(5 \,
m^2 + 8 \, m + 3\right)} x^{3} + 2 \, {\left(16 \, m^2 +
21 \, m + 7\right)} x^{2}\\ &- 2 \, {\left(10 \, m^2 + 11 \,
m + 3\right)} x + 4 \, m^2 + 4 \, m + 1, \quad \mbox{ and }
\end{split}\]
\[\begin{split}
p_2 = \; &  4 \, {\left(m^2 + 2 \, m + 1\right)} x^{4} - 4 \, {\left(3 \,
m^2 + 8 \, m + 5\right)} x^{3} + 2 \, {\left(7 \, m^2 +
21 \, m + 16\right)} x^{2}\\ &- 2 \, {\left(3 \, m^2 + 11 \,
m + 10\right)} x + m^2 + 4 \, m + 4,
\end{split}
\]
where $x$ is one of  $s_{13}, s_{14}, s_{23}$ or $s_{24}$.
For the case when $s_{12}$ or $s_{34}$ is the only variable not eliminated, we obtain the products
\[ 
(m s_{12} + s_{12} - 1)(m s_{12} + s_{12} - 2 m - 1), \quad \mbox{ or }
\]
\[ 
(m s_{34} + s_{34} - m)(m s_{34} + s_{34} - m - 2),
\]
respectively.
All the  Gr\"obner basis  computations were performed using {\sc Singular} \cite{decker_singular_2011} and  Sage \cite{stein_sage_2011}.

Analyzing the polynomials $p_1$ and $p_2$ (which we do in the next section), one can prove the following:

\begin{lemma}
\label{lemma:roots}
For any $m \in (-1,1),$ the polynomial $p_1$ has no real positive roots.
For any $m \in (-1,1),$ the polynomial $p_2$ has four positive distinct real roots
except for $m=0$, where one of the roots is repeated (at 1).
If $m \in [0,1]$, each root of $p_2$ lies in one of the intervals   $J_1=[0,\frac 1 2], J_2=[\frac 1 2, 1], J_3=[1, \frac 3 2]$, and $J_4=[\frac 3 2,5 ]$.   If $m \in (-1,0)$, each 
 root of $p_2$ lies in one of the intervals $K_1=[0,\frac 1 2],K_2=[\frac 1 2,1],
 K_3=[1, \frac{m+2}{m+1}]$ and $K_4=[\frac{m+2}{m+1},\infty]$.  
\end{lemma}

Using Lemma~\ref{lemma:roots}, we obtain the following fundamental result:
\begin{theorem}
\label{theo:convexity}
Let ${\bf x}$ be a strictly planar relative equilibrium of the four-vortex problem with vorticities
$\Gamma_1 = \Gamma_2 = 1$ and $\Gamma_3 = \Gamma_4 = m$.
\begin{enumerate}

\item  For any $m>0$, every convex relative equilibrium ${\bf x}$ has a line of symmetry,
and the only possible strictly planar asymmetric configurations are concave with the 
two vortices of smaller strength lying on the exterior triangle.  If $m=1$, all
solutions contain a line of symmetry.

\item  For any $m<0$, every concave relative equilibrium ${\bf x}$ has a line of symmetry, 
and the only possible strictly planar asymmetric configurations are convex 
with equal-strength vortices necessarily adjacent.
\end{enumerate}

\end{theorem}

\begin{proof}
First note that the case $\lambda' > 0$ is necessarily excluded in the hypotheses of the theorem.
If $m > 0$, $\lambda' = - \lambda/\Gamma < 0$ is guaranteed.  If $m < 0$ and $\lambda' > 0$, then
the configuration must be convex.  Hence, we can assume that $\lambda' = -1$.

If we compute a Gr\"obner basis of $\tilde G_{I_s}$ that eliminates 
all the $s_{ij}$ variables except  $s_{12}$ and $s_{13}$, we obtain several polynomials including 
\[
p_3= {\left(m s_{12} - 2 \, m + s_{12} - 1\right)}{\left(m s_{12} + s_{12} - 1\right) }
\]
and
 \begin{equation}\label{eq:p_4}
\begin{split}
p_4=&-8 \, m s_{13}^{4} + 8 \, {\left(m s_{12} + 3 \, m - s_{12}
+ 1\right)} s_{13}^{3}  - 2 \, {\left(9 \, m s_{12} - 4 \, s_{12}^{2}
+ 14 \, m - s_{12} + 5\right)} s_{13}^{2}\\
&+ 2 \, {\left(7 \, m s_{12} - 4 \, s_{12}^{2} + 6 \, m + s_{12} + 3\right)} s_{13}
 - 3 \,m s_{12} + 2 \, s_{12}^{2} - 2 \, m - s_{12} - 1.
\end{split}
\end{equation}
The roots of $p_3$ are $\sigma_1=\frac{2m+1}{m+1}$ and $\sigma_2=\frac{1}{m+1}$.
Substituting $\sigma_1$ into $p_4$, clearing denominators, and rescaling by a constant gives the value $p_1(s_{13})$,
which is nonzero for $m \in (-1,1)$ by the first statement in Lemma~\ref{lemma:roots}. 
Substituting $\sigma_2$ into $p_4$, clearing denominators, and rescaling by a constant gives the value $p_2(s_{13})$. 
For the case $m \in (-1,1)$,
$p_2$ has four distinct  positive real roots. 
Thus, in order to have a geometrically realizable solution, we must have $s_{12} = 1/(m+1)$ and
$s_{13}$ must be one of the four positive roots of $p_2$.
From equation~(\ref{sols-s12s34}), we also have that $s_{34} = (m+2)/(m+1)$.

Since $\lambda' = -1$, we have $1/\sqrt{- \lambda'} = 1$ in Propositions \ref{prop:concave} and~\ref{prop:convex}.
If $m \in (0,1)$, then $s_{12} < 1$ and $s_{34} > 1$.  Since we have saturated with respect to all possible differences
of the remaining four $s_{ij}$ variables, the values of $s_{13}, s_{14}, s_{23}, s_{24}$ must all be {\em distinct} roots of $p_2$.
By Lemma~\ref{lemma:roots},  two of these values are less than one, and two are greater than one. 
By Proposition~\ref{prop:convex} part 1., the configuration cannot be convex.  Since we have saturated
the Gr\"{o}bner basis to eliminate symmetric solutions, it follows that for $m \in (0,1)$, the only strictly planar, asymmetric
configurations are concave.  By Proposition~\ref{prop:concave} part 1., 
vortices 3 and 4 (with equal strength $m$) lie on the outer triangle, while either vortex 1 or 2
can lie in the interior of the concave configuration.  (The situation is reversed under the transformation
$m \mapsto \frac{1}{m}$.) 

For the special case $m=1$, we have that $s_{12} = 3/2$ and $s_{34} = 1/2$
or vice versa.  If $s_{12} = 3/2 = \sigma_1$, then substitution into $p_4$ yields that $s_{13}$
is a root of $p_1$.  But the only real roots of $p_1$ when $m=1$ are $1/2$ and $3/2$.
Thus, we either have $s_{12} = s_{13}$ or $s_{34} = s_{13}$ and by equation~(\ref{eq:symm1}),
the configuration is necessarily a kite (a symmetric configuration).  If $s_{12} = 1/2 = \sigma_2$,
a similar argument with $p_2$ replacing $p_1$ also yields a kite configuration.
This completes the proof of part 1. of the theorem.

If $m \in (-1,0)$, we have  $s_{12 } > 1$, $s_{34} > 1$.  Taken together with 
Lemma~\ref{lemma:roots}, this implies that four mutual distances are greater than one and two are less than one.  
Therefore, from Propositions \ref{prop:concave} and~\ref{prop:convex}, it follows that the only real, positive
solutions to the system of equations given by $\tilde G_{I_s}$ correspond to convex asymmetric configurations 
where the equal-strength vortices are adjacent.  This proves part 2. of the theorem.
\end{proof}

\subsection{Proof of Lemma~\ref{lemma:roots}}

In this section we analyze the polynomials $p_1$ and $p_2$ and prove Lemma~\ref{lemma:roots}. 
For $p_1$, we use Lemma~\ref{lemma:quartic-roots}, while for $p_2$,
we make appropriate choices of M\"obius transformations.


Both $p_1$ and $p_2$ are polynomials of degree four.   
Moreover, if $m > 0$, their coefficients have four sign changes, and thus they have either four, two or zero positive real roots, 
according to Descartes' rule of signs. The roots of the 
polynomials can be obtained using Ferrari's formula for quartic equations \cite{cardano_1545}.  However, we can understand 
a great deal about the solutions by using the resolvent cubic and M\"obius transformations.

One important observation is that if we make the change $m \mapsto \frac{1}{m}$ in the polynomial $p_1$, 
and clear the denominators, we obtain $p_2$.  In particular, this means that the roots of $p_2$ for an $m \in (0,1)$ are
equivalent to the roots of $p_1$ for $m \in (1,\infty)$.   Hence, it suffices to study the polynomials for $|m|<1$.

\subsubsection*{Discriminants of $p_1$ and $p_2$}

The discriminant of $p_1$ is $256 (m+1)^4 (5m+3)^2 (2m+1)^2  (m-1)^2$, which is strictly positive except for the $m$-values
$-1, -3/5, -1/2$ and $1$ where it vanishes.  At $m=-1$, $p_1$ reduces to a quadratic polynomial with
repeated roots at $x = 1/2$.  For $m = -3/5$, $p_1$ has only complex roots, while for $m = -1/2$,
$p_1$ has a double root at $x=0$.  At $m=1$, there are double real roots at $x=1/2$ and 
$x = 3/2$.
 
The discriminant for $p_2$ is $256 (m + 2)^2 (3m + 5)^2 (m + 1)^4  m^2 (m - 1)^2$
which is strictly positive except for the $m$-values
 $-2, -5/3, -1, 0$ and $1$ where it vanishes.  Focusing on the values between $-1$ and $1$, we have that 
 at $m=-1$, $p_2$ reduces to a quadratic polynomial with
 repeated roots at $x = 1/2$.  For $m = 0$, $p_2$ has a double root at $1$ and two other roots
 at $(3 \pm \sqrt{5} \, )/2$.  At $m=1$, $p_2$ has double real roots at $x=1/2$ and 
 $x = 3/2$.  For all values of $m \in (-1,1)$, except for $m=0$, the four roots of $p_2$ are
 distinct.

\subsubsection*{Polynomial $p_1$}

After applying the shift $x \mapsto  x + (m+1)(5m+3)/(4(m+1)^2)$ to remove the
 cubic term in $p_1$, we compute the key coefficients of the
 shifted quartic to be 
 \begin{eqnarray*}
 p & =&   - \frac{11m^2 + 6m - 1}{8 (m+1)^2}\; , \quad \mbox{ and } \\[0.05in]
 p^2 - 4r & = &   \frac{ (m - 1) (5m + 3) (7m + 3)}{ 16(m+1)^3 }\; .
\end{eqnarray*}
It is straight-forward to check that either $p > 0$ or $p^2 - 4r < 0$ (or both), for each $m$-value in $(-1,1)$.
Using Lemma~\ref{lemma:quartic-roots}, it follows that the roots of $p_1$ are all complex for $m \in (-1,1)$, except
when $m = -1/2$, where 0 is a double root.

\subsubsection*{Polynomial $p_2$: The case $m \in [0,1]$}


Consider the intervals $J_1=[0,\frac 1 2], J_2=[\frac 1 2, 1], J_3=[1, \frac 3 2]$, and $J_4=[\frac 3 2,5 ]$.
We show that for each $m \in [0,1]$, $p_2$ has a root in each of these intervals.  The result
follows from direct computation for $m=0$ and $m=1$, noting that some of the roots
are a shared endpoint of adjacent intervals.

\begin{center}
\begin{tabular}{l|l}
\hline
$x$-interval& M\"obius Transformation\\
\hline 
  &  \\
$J_1$&  $x=\frac 1 2\frac{u}{u+1},\quad m=\frac{\alpha}{\alpha+1}$\\[0.07in]
$J_2$&  $x=\frac{u+\frac 1 2 }{u+1},\quad m=\frac{\alpha}{\alpha+1}$\\[0.07in] 
$J_3$&  $x=\frac{\frac 3 2 u+1 }{u+1},\quad m=\frac{\alpha}{\alpha+1}$\\[0.07in]
$J_4$&  $x=\frac{5u+\frac 3 2 }{u+1},\quad m=\frac{\alpha}{\alpha+1}$\\
\end{tabular} 
\end{center}

Using the  M\"obius transformations given in the table above, after clearing the denominators,  we obtain the following four polynomials 
{\footnotesize\begin{align*}  
P_{J_1}=&  -{\left(2 \, \alpha + 1\right)} u^{4} - 2 \, {\left(2 \, \alpha + 1\right)} u^{3}
+ 2 \, {\left(8 \, \alpha^{2} + 11 \, \alpha + 4\right)} u^{2} + 4 \, {\left(12 \,
\alpha^{2} + 17 \, \alpha + 6\right)} u + 36 \, \alpha^{2} + 48 \, \alpha + 16  \\
P_{J_2} = \; & 4 \, \alpha^{2} u^{4} + 4 \, {\left(4 \, \alpha^{2} + \alpha \right)} u^{3} + 2 \,
{\left(8 \, \alpha^{2} - \alpha - 2\right)} u^{2} - 6 \, {\left(2 \, \alpha + 1\right)}
u - 2 \, \alpha - 1    \\
P_{J_3}=& -{\left(18 \, \alpha + 5\right)} u^{4} - 2 \, {\left(22 \, \alpha + 5\right)}
u^{3} + 2 \, {\left(8 \, \alpha^{2} - 13 \, \alpha - 2\right)} u^{2} + 4 \alpha \,
{\left(4 \, \alpha - 1 \right)} u + 4 \, \alpha^{2}    \\
P_{J_4}= \; & 4 \, {\left(3969 \, \alpha^{2} + 3352 \, \alpha + 704\right)} u^{4} + 4 \,
{\left(1764 \, \alpha^{2} + 815 \, \alpha + 16\right)} u^{3} + 2 \, {\left(392 \,
\alpha^{2} - 397 \, \alpha - 218\right)} u^{2}\\
&- 2 \, {\left(134 \, \alpha + 45\right)}
u - 18 \, \alpha - 5.   
\end{align*}}
Using Descartes' rule of signs, it is straight-forward to show that 
for any $\alpha > 0$, 
each of the polynomials above has precisely one positive real root.  The values at the endpoints 
of each $J_i$ can be determined by direct substitution.
It follows that $p_2=0$ has one solution in each of the intervals $J_1,J_2, J_3$ and $J_4$.  

\subsubsection*{Polynomial $p_2$: The case $m \in (-1,0)$}

Consider the intervals $K_1=[0,\frac 1 2], K_2=[\frac 1 2, 1], K_3=[1, \frac{m+2}{m+1}]$, and $K_4=[\frac{m+2}{m+1},\infty]$.
We show that for each $m \in (-1,0)$, $p_2$ has a root in each of these intervals. 

\begin{center}
\begin{tabular}{l|ll}
\hline
$x$-interval& \multicolumn{2}{|c}{M\"obius Transformation}\\
\hline\\[-0.1cm]
$K_1$&  $x=\frac 1 2\frac{u}{u+1}$,&$m=-\frac{1}{\alpha+1}$\\[0.15cm]
$K_2$&  $x=\frac{u+\frac 1 2 }{u+1}$,&$m=-\frac{1}{\alpha+1}$\\ [0.15cm]
$K_3$&  $x=\frac{1+\frac{m+2}{m+1}u }{u+1}$,&$ m=-\frac{1}{\alpha+1}$\\ [0.15cm]
$K_4$&  $x=u+\frac{m+2}{m+1}$,& $m=-\frac{1}{\alpha+1}$\\ 
\end{tabular} 
\end{center}

{\footnotesize\begin{align*}  
P_{K_1}= \; &(\alpha^2+2\alpha)u^4+2(\alpha^2+2\alpha)u^3-2(4\alpha^2+5\alpha+2)u^2\\
        &-2(12\alpha^2+14\alpha+4)u-4(4\alpha^2+4\alpha+1)\\
P_{K_2}=&-4u^4+4(\alpha-2)u^3+2(2\alpha^2+7\alpha-2)u^2+6(\alpha^2+2\alpha)u+(\alpha^2+2\alpha)\\
P_{K_3}= \; &4(4\alpha^4+12\alpha^3+9\alpha^2+2\alpha)u^4+8(6\alpha^4+17\alpha^3+11\alpha^2+2\alpha)u^3\\
        &+8(2\alpha^4+2\alpha^3-10\alpha^2-9\alpha-2)u^2-8(\alpha^3+5\alpha+2)u-4\alpha^2\\
P_{K_4}= \; &-16\alpha^3u^4-8(6\alpha^3+4\alpha^2)u^3-8(4\alpha^3+5\alpha^2+2\alpha)u^2\\
        &+8(2\alpha^3+5\alpha^2+2\alpha)u+4(4\alpha^3+12\alpha^2+9\alpha+2).
\end{align*}}

Using Descartes' rule of signs, it is straight-forward to show that for any $\alpha > 0$, each of the polynomials above has one positive real root. 
Again, it is traightforward to check the 
behavior at the endpoints of each interval by direct substitution.  It follows that $p_2=0$ has one solution in each of the intervals $K_1,K_2, K_3$ and $K_4$. 
\hfill $\Box$

\subsection{The variety of asymmetric configurations}

\label{section:AsymmetricCC}

In this section we study the asymmetric solutions to our problem and show that there are exactly eight asymmetric solutions
for each $m \in (-1,1)$.  We restrict to the case $\lambda' = -1$ since $\lambda' > 0$ only leads to symmetric
solutions (as shown in Section~\ref{sec:kites-lampos}.)
We first recall some definitions and theorems from algebraic geometry (see \cite{bochnak_real_1998}  and~\cite{cox_ideals_2007} for more details).

\begin{definition}
 Let $A$ be a commutative ring. 
\begin{enumerate}
\item An ideal $I$ of $A$ is said to be real if, for every sequence $a_1,\ldots, a_p$ of elements of~$A$, 
we have
\[
a_1^2+\ldots a_p^2\in I\implies a_i\in I, \mbox{ for } i=1,\ldots,p.
\]
\item $\sqrt[R]{I}$ is  the smallest real ideal of $A$ containing $I$ and is called  the real radical of the ideal~$I$.
\end{enumerate}
\end{definition}
Let $k$ be a field. If $I\subset k[x_1,\ldots,x_n]$ is an ideal, we denote by  $\V(I)$  the set 
\[
\V(I)=\{(a_1,\ldots, a_n)\in k^n:f(a_1,\ldots,a_n)=0 \mbox{ for all } f\in I\}.
\]
$\V(I)$ is an affine variety. In particular if $I=<f_1,\ldots,f_s>$, then $\V(I)=\V(f_1,\ldots,f_s)$. 
If $k$ is the field of the real numbers ${\mathbb R}$ we will say $\V(I)$ is a  real algebraic variety. 
Note that while this terminology is common in algebraic geometry books, it is different from the terminology frequently 
used in real algebraic geometry  (see~\cite{bochnak_real_1998}, for example).  
\begin{definition}
Let $V\subset k^n$ be an affine variety. Then we set
\[\I(V)=\{f\in k[x_1,\ldots,x_n]:f(a_1,\ldots a_n)=0 \mbox{ for all } (a_1,\ldots a_n)\in V\}.\]
$\I(V)$ is an ideal and it is called the ideal of the variety $V$.
\end{definition}
We are now ready to state a version of the Real Nullstellensatz:
\begin{theorem}[Real Nullstellensatz]
Let $k$ be a real closed field and $I$ an ideal of $k[x_1,\ldots,x_n]$. Then $\I(\V(I))=\sqrt[R]{I}$.
\end{theorem}

The dimension of an affine variety $V\subset k^n$, denoted $\mbox{dim}V$, is the degree
of the affine Hilbert polynomial of the corresponding ideal $\I(V ) ⊂ k[x_1,\ldots,x_n ]$ 
(see \cite{cox_ideals_2007} for more details). This degree can be easily computed using {\sc Singular}.
In the case $k={\mathbb R}$ one needs to know the ideal of the variety, or by the Real Nullstellensatz, the real radical of the ideal.
The real radical  can be computed using the {\tt realrad.lib} \cite{spang_tt_2011} library of {\sc Singular} \cite{decker_singular_2011}. 
More details about the algorithms can be found in the paper \cite{spang_zero-dimensional_2008}.
We now apply the above theory to our problem.

\begin{theorem}
Consider the ring ${\mathbb R}[s_{12},s_{13},s_{14},s_{23},s_{24},s_{34},m]$. The asymmetric relative equilibria 
configurations form a one-dimensional real variety $V\subset {\mathbb R}^7$.
\end{theorem}

\begin{proof}
Consider the Gr\"obner basis $G_{I_s}$ for the ideal $I_s$  and saturate with respect to $s_{13}, s_{14}, s_{24},$
$(s_{13}-s_{24})$, $(s_{14}-s_{23})$, $(s_{13}-s_{14})$, $(s_{23}-s_{24})$, $(s_{13}-s_{23})$, and $(s_{14}-s_{24})$, as before.
In addition, saturate with respect to $(m s_{12} - 2m + s_{12} - 1)$ to eliminate the 
case $s_{12} = (2m+1)/(m+1)$ for which the solutions are complex.  We obtain the following polynomials:
\[\begin{split}
&f_1=s_{13} + s_{14} + s_{23} + s_{24} - 2s_{34} - 1\\
&f_2= s_{12} - s_{34} + 1\\
&f_3= s_{34}m  + s_{34} - m -2\\ 
&f_4=2s_{24}^2 - s_{14}s_{34} - s_{23}s_{34} - 4s_{24}s_{34} + 2s_{34}^2 + 2s_{23}\\
&f_5= 2s_{23}s_{24}- 2s_{23} - 2s_{24} + s_{34}\\
&f_6= 2s_{14}s_{24} - s_{34} \\
&f_7=2s_{23}^2 + s_{14}s_{34} - 3s_{23}s_{34} +2s_{24} - s_{34}\\
&f_8= 2s_{14}s_{23} - s_{14}s_{34} - s_{23}s_{34} + s_{34}\\
&f_9=2s_{14}^2 - 3s_{14}s_{34} +s_{23}s_{34} - 2s_{14} - 2s_{23} - 2s_{24} + 3s_{34} + 2.
\end{split}\]
Let $I=<f_1,\ldots, f_9>$. We want to find the dimension of  the real variety $\V(I)$. 
First, we find $\I(\V(I))$ or $\sqrt[R]{I}$. 
This can be computed using the {\tt realrad.lib} \cite{spang_tt_2011} library of {\sc Singular} \cite{decker_singular_2011}.
However in this case it turns out that $\sqrt[R]{I}=I$, and hence $f_1,\ldots,f_9$ are the generators of the real radical of $I$.
The degree of the affine  Hilbert polynomial of $\sqrt[R]{I}=I$ is one, and thus $\V(I)$ is a one dimensional real algebraic variety. 
\end{proof}

We now want to see if the variety contains singular points. First we recall some definitions and theorems. 

\begin{definition}
 Let $V\subset k^n$ be an affine algebraic variety then the Zariski tangent space of $V$ at $p=(p_1,\ldots, p_n)$
denoted by $T_p(V)$ is the variety
\[
T_p(V)=\V\left(d_p(f): f\in\I(V)\right)
\]
where  $d_p(f)=\sum_{i=1}^n\frac{\partial f}{\partial x_i}(p)(x_i-p_i)$.
\end{definition}

Clearly if $\I(V)=<f_1,\ldots,f_s>$ then $T_p(V)=\V(d_p(f_1),\ldots, d_p(f_s))$, and it is the translate of a linear subspace 
of $k^n$.
Recall that if $V\subset k^n$ is an affine variety, then $V$ is irreducible if and only if ${\bf I}(V)$ is a prime ideal. 
Thus we have the following definition:

\begin{definition}\label{def:singular}
 Let $V\subset k^n$ be an irreducible variety. A point $p$ in $V$ is nonsingular (or smooth) provided that $\dim T_p(V)=\dim V$.
In other words, if $\I(V)=<f_1,\ldots,f_s>$, then $p$ is nonsingular if and only if the rank of the  Jacobian matrix
 $J=\left[ \frac{\partial f_i}{\partial x_j}\right]$ is equal to $n-\dim(V)$.
Otherwise, $p$ is a singular point of $V$.
\end{definition}
The notion of Zariski tangent space is defined even at a singular point. However, when $V\subset{\mathbb R}^n$ is an irreducible variety
and $p\in V$ is a nonsingular point, a neighbourhood of $p$ in $V$ is a $C^\infty$ submanifold of ${\mathbb R}^n$ and the tangent 
space of $V$ at $z$ (in the $C^\infty$ sense) coincide with the Zariski tangent space. If $p$ is singular, then the dimension of 
$T_p(V)$ is bigger than the dimension of $V$.

\begin{theorem}
 The variety of the asymmetric strictly planar relative equilibria configurations has no singular points and hence it is 
a smooth one-dimensional manifold. 
\end{theorem}

\begin{proof}
The Jacobian matrix of $(f_1,\ldots, f_9)$ is 
{ \small
\[
\begin{bmatrix}
0&1&1&1&1&-2&0\\
1&0&0&0&0&-1&0\\
0&0&0&0&0&{ m}+1&{ s_{34}}-1\\
0&0&-{ s_{34}}&-{ s_{34}}+2&4\,({ s_{24}}-{ s_{34}})&-{ s_{14}}-{ s_{23}}-4({ s_{24}}-{ s_{34}})&0\\ 
0&0&0&2\,{ s_{24}}-2&2\,{ s_{23}}-2&1&0\\ 
0&0&2\,{ s_{24}}&0&2\,{ s_{14}}&-1&0\\ 
0&0&{ s_{34}}&4\,{ s_{23}}-3\,{ s_{34}}&2&{ s_{14}}-3\,{ s_{23}}-1&0\\
0&0&2\,{ s_{23}}-{ s_{34}}&2\,{ s_{14}}-{ s_{34}}&0&-{ s_{14}}-{ s_{23}}+1&0\\
0&0&4\,{ s_{14}}-3\,{ s_{34}}-2&{ s_{34}}-2&-2&-3\,{ s_{14}}+{ s_{23}}+3&0\end {bmatrix}  
\]
}
Using Gaussian elimination we obtain the following matrix

\[
\left[
\begin{BMAT}(rc){c:c}{c:c}
\begin{BMAT}(rc){cccc}{cccc}
1 & 0&0&0\\
0 & 1&1&1\\
0 & 0&{ s_{34}}&4\,{ s_{23}}-3\,{ s_{34}}\\
0 & 0&0&2s_{24}-2
\end{BMAT}&
\begin{BMAT}(rc){ccc}{cccc}
0 & -1&0\\
1 & -2&0\\
2 & s_{14}-3s_{23}-1&0\\
2s_{23}-2 & 1&0
\end{BMAT}\\
\begin{BMAT}(rc){cccc}{ccccc}
0 & ~0~&0&~~~~~~~0~~~~~~~\\
0 & 0&0&0\\
0 & 0&0&0\\
0 & 0&0&0\\
0 & 0&0&0
\end{BMAT}&
\begin{BMAT}(rc){c}{c}
\mbox{\Huge $A$}
\end{BMAT}
\end{BMAT}
\right],
\]

where 
\[A=
\begin{bmatrix}
q_1&q_2&0\\
0&m+1&s_{34}-2\\
0&a&0\\
0&b&0\\
0&c&0\\ 
\end{bmatrix}
\]
and $a,b,c$  are polynomials that belong to the ideal of the variety (and hence  vanish on the variety), and 
\[\begin{split}
&q_1=-4(s_{23}-s_{24})(2s_{23}+2s_{24}-2s_{34}-1),\\
&q_2=-8\,{ s_{23}}\,{ s_{24}}-8\,{{ s_{24}}}^{2}+8\,{ s_{24}}\,{ s_{34}}+6
\,{ s_{24}}+4\,{ s_{23}}-4\,{ s_{34}}.\\   
\end{split}
\]
Since $s_{34}\neq 0$, $s_{24}-1\neq 0$, $q_1\neq 0$, and $m_3-1$ and $s_{34}-2$ are not both zero on the variety, it follows that
the rank of the Jacobian matrix is $6$. Since $\dim (V)=1$, by Definition~\ref{def:singular} the variety has no singular points. 
It follows that it is a smooth manifold. 
\end{proof}

 \begin{figure}[t]
 \centering
 \includegraphics[width=10 cm,keepaspectratio=true]{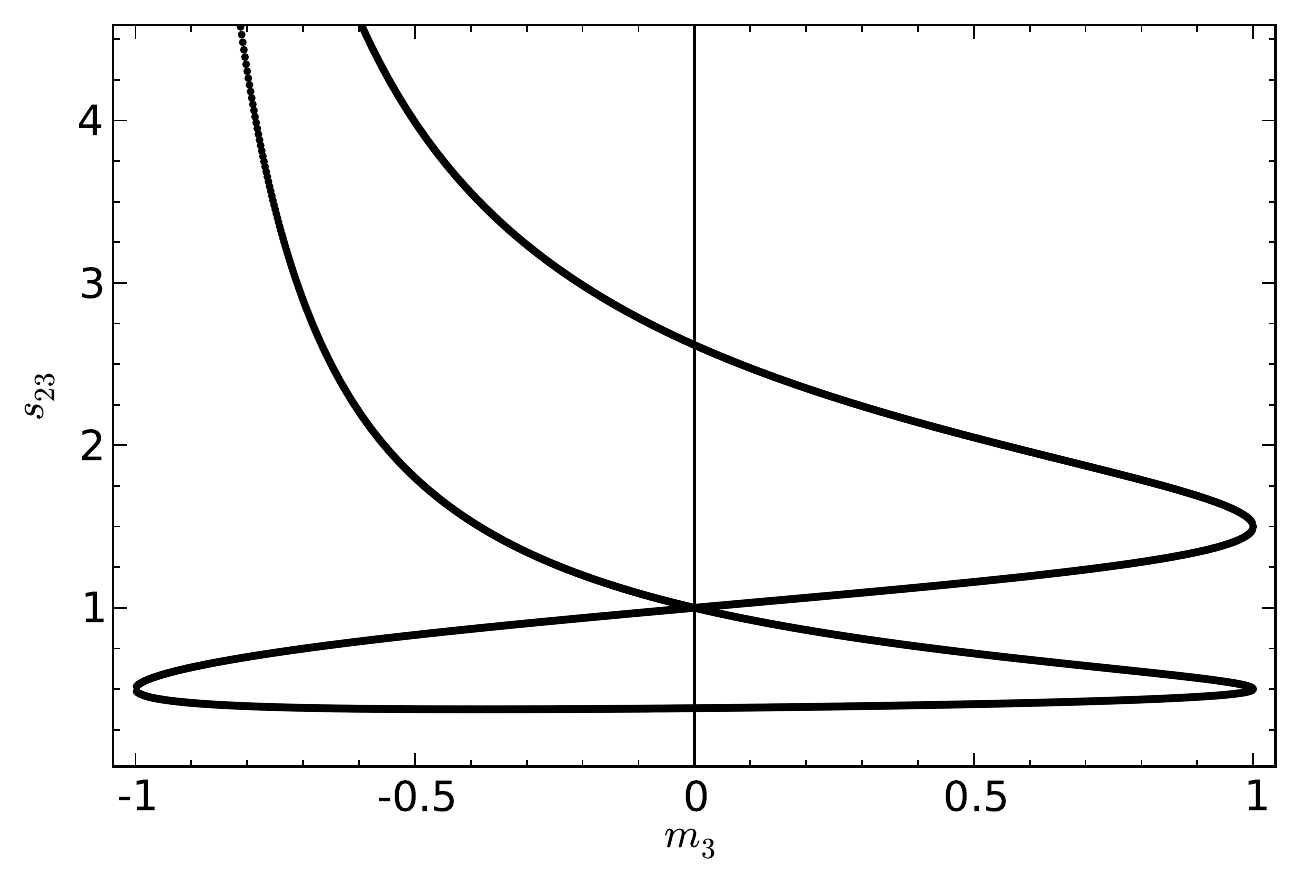}
 \caption{ Projection of part of the manifold of  asymmetric configurations onto the plane $(m,s_{23})$. 
 As $m\rightarrow -1$,  $s_{23} \rightarrow \infty$ along two branches while $ s_{23} \rightarrow \frac 1 2$
 along the other two branches. }
 \label{fig:manifold} 
\end{figure}

\begin{theorem}
 For each value of $m$ in $(-1,1)$, there are exactly four (eight counting reflected solutions) strictly planar asymmetric relative equilibria configurations. 
They are convex if $m < 0$ and concave if $m > 0$.
\label{theorem:asymm-cc}
\end{theorem}

\begin{proof}
The proof consists of showing that the system of equations $f_1=0,\ldots, f_9=0$ has four solutions for each value of $m$ in $(-1,1)$.
The convexity  of the configuration then follows from Theorem~\ref{theo:convexity}.
We use an algorithm for the triangular decomposition of semi-algebraic systems. Such an algorithm, given a system of
equations and inequalities $S$, computes simpler systems $S_1, ... , S_k$
such that a point is a solution of the original system $S$ if and only if it is a solution of one of the systems $S_1, ... , S_k$. Each of these systems has a triangular shape and remarkable properties: for this reason it is called a regular semi-algebraic system and the set of the $S_1, ... , S_k$ is called a full triangular decomposition of $S$. See \cite{chen_triangular_2010} and references therein for some  background.
The algorithm  is detailed in \cite{chen_triangular_2010} and it is available in   Maple${15}^\mathrm{TM}$ via  the {\tt RealTriangularize} command of the {\tt RegularChains} package.
A full triangular decomposition of the system $f_1=0, \ldots, f_9=0$ is given below.

If $m=0$ we have the following two systems:
\begin{align*}
&s_{12}-1=0    & & s_{12}-1=0 \\
&s_{13}+s_{23}-3=0 &  &s_{13}-1=0 \\
&s_{14}-1=0     & & s_{14}+s_{24}-3=0    \\
&s_{23}^2-3s_{23}+1=0 & & s_{23}-1=0 \\
&s_{24}-1=0 & & s_{34}^2-3s_{24}+1=0\\
&s_{34}-2=0 & & s_{34}-2,
\end{align*}
if $m=1$ we have: 
\begin{align*}
&2s_{12}-1=0    & & 2s_{12}-1=0 \\
&2s_{13}-1=0    & & 2s_{13}-3=0 \\
&2s_{14}-1=0    & & 2s_{14}-3=0 \\
&2s_{23}-3=0    & & 2s_{23}-1=0 \\
&2s_{24}-3=0    & & 2s_{24}-1=0 \\
&2s_{34}-3=0    & & 2s_{34}-3=0 \\
\end{align*}
and if $m \in (-1,0)\cup(0, 1)$  we obtain:
\begin{align*}
&s_{12}-s_{34}+1=0\\
&s_{13}+s_{14}+s_{23}+s_{24}-2s_{34}-1=0\\
&s_{34}s_{14}+(s_{34}-2)s_{23}-2s_{24}^2+4s_{24}s_{34}-2s_{34}^2=0\\
&(2s_{24}-2)s_{23}-2s_{24}+s_{34}=0\\
&4s_{24}^4+(-8s_{34}-4)s_{24}^3+(4s_{34}^2+6s_{34}+4)s_{24}^2+(-4s_{34}^2-2s_{34})s_{24}+s_{34}^2=0\\
&(m+1)s_{34}-m-2=0
\end{align*}
This decomposition immediately shows that the system has no real solutions if $m = -1$, two real solutions if $m = 1$, and 
four real solutions if $m = 0$. It remains to study the case $m \in(-1,0)\cup(0, 1)$. 
Solving the last equation of the last system of the triangular decomposition yields $s_{34}=\frac{m+2}{m+1}$.
Substituting this into the preceding equation of the same system, after clearing the denominators, yields $p_2(s_{24})$. From Lemma~\ref{lemma:roots}
we know that $p_2$ has four distinct positive real roots for $m \in (-1,0) \cup (0, 1)$. If we fix one such solution for $s_{24}$ 
and substitute into the first four equations of the system, the resulting system is
$Ax=b$, where $x=(s_{12},s_{13},s_{14},s_{23})$,
\[A=
\begin{bmatrix}
 1&0&0&0\\
 0&1&1&1\\
 0&0&s_{34}&(s_{34}-2)\\
 0&0&0&2(s_{24}-1)
\end{bmatrix}
\]
and $b=[-1+s_{34},-s_{24}+2s_{34}+1,s_{24}^2-4s_{24}s_{34}+2s_{34}^2,2s_{24}-s_{34}]^T$. 
Since the matrix $A$ is invertible, $Ax=b$ has exactly one solution for every vector $b$.
This proves the theorem since it implies that the system $f_1=0,\ldots f_9=0$ has four real positive solutions.

\end{proof}
\begin{remark}
 Note that for $m=1$ we obtain two solutions which correspond to an equilateral triangle configuration with a vortex
 at the center.  These are symmetric 
configurations that have not been eliminated because in the case of equal vorticities, there is an
additional symmetry that we did not exclude when saturating. Also, note
that the polynomials $f_1,\ldots, f_9$ were obtained saturating with respect to $(ms_{12}-2m+s_{12}-1)$. This eliminated the real 
solutions of the system $f_1=0,\ldots, f_9=0$ for $m \in (-\infty,-1)\cup(1,\infty)$.  However, we know that the number of real
 (positive solutions) of the original system (e.g., before saturating) for $m \in (-\infty,-1)\cup(1,\infty)$ is the same as the number of solutions of 
$f_1=0,\ldots, f_9=0$ for $m \in (-1,1)$.
\end{remark}

\section{Symmetric Strictly Planar Relative Equilibria}
\label{section:symmetric}

In this section we investigate all possible symmetric, strictly planar, relative
equilibria in the four-vortex problem with $\Gamma_1 = \Gamma_2 = 1$ and $\Gamma_3 = \Gamma_4 = m$.
The four possible configurations are an isosceles trapezoid, a concave kite, a convex kite, or
a rhombus.   We separate the kite configurations into two cases based on the sign of
$\lambda'$.


\subsection{The isosceles trapezoid family}
\label{isos-trap}

Without loss of generality, suppose that the vortices
are ordered sequentially around an isosceles trapezoid, so that the lengths of the diagonals are given by $r_{13} = r_{24}$
and the congruent legs have length $r_{14} = r_{23}$.  Due to the symmetry of the configuration, the four vortices lie
on a common circle, that is, the isosceles trapezoid is a cyclic quadrilateral.  By Ptolemy's theorem, we have that
\begin{equation}
r_{13}^2 \; = \; r_{24}^2 \; = \; r_{12} r_{34} + r_{14}^2.
\label{eq:Ptolemy}
\end{equation}
A choice of six mutual distances satisfying $r_{13} = r_{24}$, $r_{14} = r_{23}$ and equation~(\ref{eq:Ptolemy}) will
also satisfy $e_{CM} = 0$.

In contrast to the Newtonian four-body problem, the isosceles trapezoid family
of four-vortex relative equilibria can be solved completely (by hand) in terms of the vortex strength $m$ 
(compare with~\cite{cors_co-circular_2012}).
Let $x = r_{34}/r_{12}$ and $y = r_{14}/r_{12}$.  Rather than assuming a specific scaling on the distances (e.g., $\lambda' = -1$), 
we will solve for $x$ and $y$ in terms of $m$.  Using the symmetries of the configuration and equation~(\ref{eq:Ptolemy}),
equation~(\ref{eq:CCfactor}) reduces to
\begin{equation}
2(y^2 - x^2)(y^2 - 1) + x(2y^2 - x^2 - 1) \; = \; 0.
\label{eq:Trapxy}
\end{equation}
This relation between $x$ and $y$ is both necessary and sufficient for the trapezoid to be a relative equilibrium.

Focusing on the first and last columns of equations~(\ref{eqDziobek2}), the first and last equations are easily satisfied due to symmetry.  
The middle four equations in~(\ref{eqDziobek2}) 
are equivalent given that~(\ref{eq:Trapxy}) holds.  They determine a formula for~$m$ given by
\begin{equation}
m \; = \;  \frac{x(1-y^2)}{y^2 - x^2},
\label{eq:TrapMass}
\end{equation}
where we have used $|A_i| = r_{jk} r_{jl} r_{kl}/(4 r_c)$ for the area of a triangle circumscribed
in a circle of radius $r_c$.  Solving equation~(\ref{eq:TrapMass}) for $y^2$ and substituting
into equation~(\ref{eq:Trapxy}) leads to the equation
\begin{equation}
x(1-x)(1+x)(x^2(2m+1) - m(m+2)) \; = \; 0.
\label{eq:TrapXSol}
\end{equation}

Since $x$ represents the ratio of two distances, there are only two possibilities.  If $x=1$, then equation~(\ref{eq:Trapxy}) quickly gives $y=1$ and the configuration
is a square.  The second equation in~(\ref{eqDziobek2}) then gives that $m=1$ and all vortices must have the same strength.  The other possibility
is that $x^2 = m(m+2)/(2m+1)$.  Substituting this value into equation~(\ref{eq:TrapMass}) and solving for $y^2$ yields
$$
y^2 \; = \;  \frac{1}{2} \left( m+2  - \sqrt{ \frac{m(m+2)}{2m + 1} } \, \right).
$$
Note that the necessary condition $x^2 > 0$ is satisfied only for $m > 0$ and $-2 < m < -1/2$.  In order for
the expression for $y^2$ to be real, we must be in one of these two regimes.  However, while $y^2 > 0$ holds
for $m > 0$, it is not satisfied on $-2 < m < -1/2$.  Therefore, there does {\em not} exist an isosceles trapezoid
solution for the case $m < 0$.   

\begin{theorem}
There exists a one-parameter family of isosceles trapezoid relative equilibria with vortex strengths
$\Gamma_1 = \Gamma_2 = 1$ and $\Gamma_3 = \Gamma_4 = m$.  The vortices 1 and~2 lie on one
base of the trapezoid, while 3 and~4 lie on the other.  Let $\alpha = m(m+2)/(2m+1)$.  
If $r_{13} = r_{24}$ are the lengths of the
two congruent diagonals, then the mutual distances are described by
$$
\left( \frac{r_{34}}{r_{12}} \right)^2 \; = \;  \alpha, \quad 
\left( \frac{r_{14}}{r_{12}} \right)^2 \; = \;  \frac{1}{2} \left( m+2  - \sqrt{ \alpha } \, \right)
$$
$$
\mbox{ and} \quad
\left( \frac{r_{13}}{r_{12}} \right)^2 \; = \;  \frac{1}{2} \left( m+2  + \sqrt{ \alpha } \, \right).
$$
This family exists if and only if $m > 0$.  The case $m=1$ reduces to the square.  For $m \neq 1$, 
the larger pair of vortices lie on the longest base.
\end{theorem}

\begin{proof}
The formula for $r_{13}^2/r_{12}^2$ comes from equation~(\ref{eq:Ptolemy}).  Substituting $m=1$ into the formulas above
quickly gives $r_{12} = r_{14} = r_{34} = r_{23} $ and $r_{13} = r_{24} = \sqrt{2} \, r_{12}$, a square.  
The only other item that remains to be shown
is the fact about the larger vortices lying on the longer base.  This follows since $\alpha > 1$ if and only
if $m > 1$.
\end{proof}

\begin{remark}
\begin{enumerate}

\item  If $0 < m < 1$, $r_{12}$ is the longest base length and the 
formulas in the theorem give $r_{34} < r_{14} < r_{12} < r_{13}$.
On the other hand, if $m > 1$, then $r_{34}$ is the longest base length and we deduce
$r_{12} < r_{14} < r_{34} < r_{13}$.  Both cases agree with the conclusion
of statement~1 in Proposition~\ref{prop:convex}.

\item  The case $m > 1$ is identical to the case $0 < m < 1$ through a rescaling of the vortex strengths
and interchanging bodies 1 and~3, and bodies 2 and~4.  Specifically, replacing $m$ by $1/m$ and
interchanging distances $r_{12}$ and $r_{34}$ leaves the three equations for the ratios of mutual distances
unchanged.  The vortex strengths map to $\Gamma_1 = \Gamma_2 = 1$ and $\Gamma_3 = \Gamma_4 = 1/m < 1$
under this transformation.

\item  As $m \rightarrow 0$, $\alpha \rightarrow 0$ and consequently $r_{34} \rightarrow 0$.  The limiting configuration is
an equilateral triangle with vortices 3 and~4 colliding.  As $m$ increases, so does $\alpha$, and the smaller base length
$r_{34}$ approaches the larger one $r_{12}$ as $m \rightarrow 1$.  The ratio of the diagonal to the larger base also increases monotonically
with~$m$.  However, the ratio of the leg to the larger base begins at $1$ ($m=0$) and decreases to a minimum value of 
approximately $0.904781$ before increasing back to $1$ at the square configuration.  This minimum value occurs
at $m \approx 0.234658$, which is the only positive root of the quartic $8m^4 + 19m^3 + 9m^2 +m - 1$.
Just as with the Newtonian case (see~\cite{cors_co-circular_2012}), the range of $y$ is surprisingly small, confined to the interval
$[0.904781,1]$.

\end{enumerate}
\end{remark}

\subsection{Kite configurations: $\lambda' < 0$}
\label{section:concave-kites}

In this section, we consider kite central configurations when $\lambda' < 0$.
Such configurations contain two vortices that are 
symmetrically located with respect to an axis of symmetry while the remaining two vortices lie on the axis of symmetry.  
The configuration formed by the vortices can either be concave or convex. 
In this section, we focus on configurations with only one line of symmetry. 
Those configurations with two lines of symmetry, i.e., the rhombus configurations, 
are studied in detail in Section~\ref{rhombus_sec}.

Our goal is to count the number of kite configurations as~$m$ varies
and describe the type of possible configurations.   
Recall that $\mathcal{\widetilde F}, \mathcal{\widetilde G}, \tilde e_{CM}$ and
$\mathcal{\widetilde H}$ represent the Albouy-Chenciner equations, the unsymmetrized 
Albouy-Chenciner equations, the Cayley-Menger determinant and the
Dziobek equations, respectively, in terms of the variables $s_{ij} = r_{ij}^2$, 
with $\lambda' = -1$.  Denote the complete system of these equations 
as~$\mathcal{\widetilde E}$.  Our goal is to prove the following theorem.

 \begin{theorem}
 \label{thm:concavekites}
 For each value of $m\in (-\infty,-2) \cup (-\frac 1 2,0) \cup (0, \infty)$, 
 system $\mathcal{\widetilde E}$ has exactly four (real, positive) solutions that correspond to  symmetric
 kite configurations. If $m\in (0,\infty)$, such configurations are concave. If $m\in (-\frac 1 2,0)$, then there are two convex 
 configurations with $\Gamma_3$ and $\Gamma_4$ on the axis of symmetry, and two concave ones with $\Gamma_1$ 
 and $\Gamma_2$ on the axis of symmetry.  If $m\in(-\infty, -2)$, the situation is reversed and there are two concave 
 configurations with $\Gamma_3$ and $\Gamma_4$ on the axis of symmetry, and two convex ones with $\Gamma_1$ 
 and~$\Gamma_2$ on the axis of symmetry.
There are no other strictly kite configurations with $\lambda' < 0$ for other values of $m$.
\end{theorem}

\begin{remark}

\begin{enumerate}
\item  As we show below, the case $m=1$ is considerably different from the Newtonian four-body problem.  In fact, 
 it is known for the Newtonian case that there are two distinct symmetric kite configurations for four equal masses:
 an equilateral triangle with a mass in the center and an isosceles triangle with a mass on the axis of symmetry~\cite{albouy_symmetric_1996}. 
 Instead, for four vortices of equal strength, the only concave kite central configuration is an equilateral triangle 
 with a vortex at the center of the triangle.

\item  Based on this theorem and Theorem~\ref{theorem:asymm-cc}, we conclude that there are {\em no} concave
relative equilibria when $m \in (-1,-1/2)$.  This is another contrast with the Newtonian four-body problem, where it is
shown in~\cite{hampton_2002} that for any choice of four positive masses, there exits a concave central configuration.

\item  The exact count on the number of concave symmetric solutions given in Table~\ref{table:count} is precisely
twice the values expressed in the theorem.  This is due to the fact that each solution found can be reflected about the axis
of symmetry, a transformation that does not change the mutual distances $r_{ij}$, but does alter the positions $x_i$.

\end{enumerate} 
\end{remark}

We have two possibilities, either vortices $3$ and~$4$ lie on the axis  of symmetry or vortices $1$ and~$2$ do.
It is convenient to study the two cases separately, since the conditions imposed by the symmetry are different.
The first case will be explored in Lemma~\ref{lemma:concavekites}, and the second in Lemma~\ref{lemma:concavekites1}.
The proof of Theorem~\ref{thm:concavekites} follows immediately from the two lemmas.

Let us consider the first case. If vortices $3$ and~$4$ lie on the axis of symmetry, then $s_{23}=s_{13}$ and $s_{24}=s_{14}$.
We compute a Gr\"obner basis for the ideal $\langle \mathcal{\bar F},\mathcal{\bar G},\bar  e_{CM},\bar {H}\rangle$,
where $ \mathcal{\bar F},\mathcal{\bar G},\bar  e_{CM},\bar {\mathcal H}$ are obtained from 
 $ \mathcal{\widetilde F},\mathcal{\widetilde G},\widetilde e_{CM}$ and $\mathcal{\widetilde H}$, respectively,
by imposing the conditions  $s_{23}=s_{13}$ and $s_{24}=s_{14}$.
Then we saturate with respect to
$s_{13}-s_{14}$, in order to exclude the rhombus configurations. We also saturate with respect to 
$s_{12},s_{13},s_{14}, s_{34}$ and $m$. This computation yields the following ten polynomials:
\begin{align*}
g_1 = & s_{34} m + s_{13} + s_{14} + s_{34} -  m - 2\\
g_2 = & s_{13} m+ s_{14} m + s_{12} + s_{13} + s_{14} + s_{34} - 3 m - 3\\
g_3=&2s_{12} s_{34} - 2 s_{13} s_{34} - 2 s_{14} s_{34} + 2 s_{34}^{2} - 
s_{12} + 2 s_{13} + 2 s_{14} + s_{34} - 3\\
g_4=&2 s_{13} s_{14} -  s_{13} - s_{14} + s_{34}\\
g_5=&2 s_{13}^{2} + 2 s_{14}^{2} - 2 s_{13} s_{34} - 2
s_{14} s_{34} + s_{12} - 4 s_{13} - 4 s_{14} + s_{34} + 3\\
g_6=&2 s_{12}s_{13} + 2 s_{12} s_{14} -  s_{12} m - 2 s_{12} -  s_{13} -  s_{14}
+ s_{34}\\
g_7=&s_{12} m^{2} + 2 s_{12}^{2} + 2 s_{13} s_{34} + 2 s_{14}
s_{34} - 2 s_{34}^{2} - 3 s_{12} m - 4 s_{12}
\\ &-  s_{13} -  s_{14} -
2 s_{34} + 2 m + 4\\
g_8=&2 s_{14}^{2} m + 2 s_{12} s_{14} + 2
s_{14}^{2} + 2 s_{14} s_{34} - 6 s_{14} m -  s_{12} + s_{13}- 5s_{14} -  s_{34}\\
&+ 2 m + 1\\
g_9=&4 s_{14}^{3} - 4 s_{14}^{2} s_{34} + 2
s_{12} s_{14} - 10 s_{14}^{2} - 2 s_{13} s_{34} + 2 s_{14} s_{34} + 2
s_{34}^{2} -  s_{12}\\
&+ s_{13} + 7 s_{14} + 2 s_{34} - 3\\
g_{10}=&4 s_{12}s_{14}^{2} - 2 s_{12} s_{14} m - 4 s_{12} s_{14} - 2 s_{14}^{2} - 2
s_{13} s_{34} + 2 s_{34}^{2} + s_{12} m + s_{12}\\
&+ 2 s_{13} + 2s_{14} + s_{34} - 3.\end{align*}
Let us denote the corresponding system of equations $g_1=0,\ldots ,g_{10}=0$ as $G$. 
Then the positive solutions of $G$ give the 
symmetric kite configurations described above. An exact count of these solutions is described in the following lemma.

\begin{lemma}
\label{lemma:concavekites}
 For each value of $m\in (0,1)$, system $G$ has exactly four real positive solutions corresponding to concave configurations. 
 For each value of $m\in (-\infty,-2) \cup (-\frac 1 2,0)$, system $G$ has exactly two real positive solutions.
 Such solutions correspond to convex configurations if  $m\in(-\frac 1 2,0)$ and to concave configurations if $m\in(-\infty,-2)$.
 For $m=1$, system $G$ has exactly two real positive solutions corresponding to an equilateral triangle
 with either vortex 3 or~4 at the center.  There are no positive real solutions for other values of~$m$.
\end{lemma}

\begin{proof}

Computing a full triangular decomposition of system~$G$ with the positivity conditions $s_{12}>0$, $s_{13}>0$, $s_{14}>0$ and $s_{34}>0$, 
we obtain five simpler systems.   If $m=1$, we have the following two systems:
\beq T_1=\begin{cases}
\begin{matrix}   2s_{12}-3=0\\ 2s_{13}-3=0 \\ 2s_{14}-1=0 \\ 2s_{34}-1=0,  \end{matrix}\end{cases}
\quad \quad T_2=\begin{cases}\begin{matrix} 2s_{12}-3=0 \\2s_{13}-1=0 \\2s_{14}-3=0\\2s_{34}-1=0.\\
              \end{matrix}
\end{cases}
\eeq
If $m=0$, we have the following two systems:
\[T_3=\begin{cases}
\begin{matrix}
 s_{12}-1=0\\
 s_{13}-1=0\\
 4s_{14}-1=0\\
 4s_{34}-3=0,\\
 \end{matrix}\end{cases}\quad
T_4=\begin{cases}
\begin{matrix}
s_{12}-1=0\\
4s_{13}-2=0\\
s_{14}-1=0\\
4s_{34}-3=0.\\
 \end{matrix}\end{cases}
\]
If $m\in(-\infty,-2)\cup(-1/2,0)\cup(0,1)$, we obtain the system
\begin{align*}T_5=\begin{cases}
&s_{12}+(-m^2-2m)s_{34}-1+m^2=0\\
&s_{13}+s_{14}+(m+1)s_{34}-2-m=0\\
&2s_{14}^2+((2m+2)s_{34}-4-2m)s_{14}+(-2-m)s_{34}+m+2=0\\
&(4+6m+2m^2)s_{34}^2+(-3-3m^2-6m)s_{34}+2m+m^2=0.
\end{cases}
\end{align*}
For all other $m$-values the simpler systems have no real solutions.

The case $m=1$ is straightforward, since $T_1$ and $T_2$ each have a unique solution corresponding to an equilateral triangle with 
either vortex 3 or~4 at the center.
The case $m=0$ is uninteresting and we do not analyze it. 
It remains  to study the solutions of system $T_5$. 

\vskip 0.2 cm
\noindent$\bullet$~{\it Solutions of $T_5$ with $s_{34}>0$ and $s_{14}>0$}
\vskip 0.2 cm

The first two equations of $T_5$ are linear in $s_{12}$ and $s_{13}$, so that to each solution $(s_{14},s_{34})$ of
 the last two equations correspond one solution $(s_{12},s_{13},s_{14},s_{34})$ of $T_5$, provided that
 the determinant of the matrix of the coefficients of the linear system is not zero.  Hence, we focus our attention 
on the last two equations of $T_5$. 
Let $A=2(2+3m+m^2)$, $B=-3(m+1)^2$, and $C=m(2+m)$. Then the last equation of $T_5$ can be written as 
\[
q_1(s_{34})=As_{34}^2+Bs_{34}+C = 0
\]
and can be viewed as a parametric equation of a parabola. Clearly $B<0$ for all $m\neq -1$, 
and $A>0$ for $m\in (-\infty,-2)\cup(-1,\infty)$. Thus $A>0$, $B<0$ on the domain where $T_5$ is valid. 
The vertex of the parabola has coordinates
\[
\left(-\frac{B}{2A},q_1\left(-\frac{B}{2A}\right)\right)=\left(\frac{3(m+1)^2}{2(2+3m+m^2)},
-\frac1 8\frac{(m-1)(m^2-4m-9)}{m+2}
\right).
\]
It is easy to see  that $-\frac{B}{2A}>0$ and   $q_1(-\frac{B}{2A})<0$  in the interval of interest, and hence the equation has 
two real solutions for each value of $m\in(-\infty,-2)\cup(-1/2,0)\cup(0,1)$.
Since $C>0$ for $m \in (-\infty,-2)\cup(0,\infty)$ and $C<0$ for $m \in (-2,0)$, it follows that the equation has two positive solutions
for $m \in (-\infty,-2)\cup(0,1)$ and only one positive solution for $m \in (-1/2,0)$.

Let us denote as $\beta_1$ and $\beta_2$ (with $\beta_1<\beta_2$) the (real) solutions of $q_1(s_{34})=0$.
Then through a standard analysis it is possible to find bounds for $\beta_1$ and $\beta_2$. Such bounds are summarized in the table below.

\renewcommand{\arraystretch}{1.5}
\begin{center}
  \begin{tabular}{| c|c|c |c|  }
 \hline
  solutions & $m \in (-\infty,-2)$ & $m \in (-\frac 1 2,0)$ & $m\in (0,1)$ \\ \hline\hline
  $\beta_1$ & $0<\beta_1<\frac 1 2$ &$\beta_1<0$& $0<\beta_1<\frac 1 2$ \\ \hline
   $\beta_2$&$\beta_2>1$ &$0<\beta_2<1 $ &$0<\beta_2< 1 $\\
    \hline
  \end{tabular}
\end{center}

Now consider the third equation of system $T_5$. Let $A'=2$, $B'=(2(m+1)s_{34}-4-2m)$, and
$C'=(m+2)(1-s_{34})$.  Then this equation takes the form 
\[
q_2(s_{14})=A's_{14}^2+B's_{14}+C'=0.
\]
Since this is a quadratic equation in $s_{14}$, there are going to be two solutions for each $s_{34}$ and $m$.
We will denote the (real) solutions corresponding to $s_{34} = \beta_1$ as $\beta_1^1$ and $\beta_1^2$ (with $\beta_1^1<\beta_1^2$), 
and the ones corresponding to  $s_{34} = \beta_2$ as $\beta_2^1$ and $\beta_2^2$ (with $\beta_2^1<\beta_2^2$).
The vertex of this parabola has coordinates
\begin{multline}
\left(-\frac{B'}{2A'},q_2\left(-\frac{B'}{2A'}\right)\right)=\\=\left(\frac{(m+1)(1-s_{34})+1}{2} ,-\frac 1 2(m+1)^2s_{34}^2+m(m+2)s_{34}-\frac m 2(m+2)\right).
\end{multline}
We can then view the ordinate of the vertex, denoted as $r(s_{34})$, as a parabola with coefficients in $m$:
\[
r(s_{34})=-\frac 1 2 (m+1)^2s_{34}^2+m(m+2)s_{34}-\frac m 2(m+2)=A''s_{34}^2+B''s_{34}+C''=0.
\]
The coordinates of the vertex of $r(s_{34})$ are
\[\left(-\frac{B''}{2A''},r\left(-\frac{B''}{2A''}\right)\right)=\left(\frac{m(m+2)}{(m+1)^2} ,-\frac 1 2\frac{m(m+2)}{(m+1)^2} \right).\]

\vskip 0.2cm
\noindent{\it CASE I: $m\in(0,1)$}
\vskip 0.2cm

Since $-\frac{B''}{2A''}>0$, $r\left(-\frac{B''}{2A''}\right)<0$, and $A''<0$ for $m\in(0,1)$, it follows that
$q_2\left(-\frac{B'}{2A'}\right)<0$ for $m\in(0,1)$. Consequently, since $C'>0$ in the same interval, $q_2=0$
has two  positive solutions corresponding to each solution of $q_1=0$ for $m\in(0,1)$.  In other words,
there are four (positive) possible values for $s_{14}$, namely  $\beta_1^1>0$, $\beta_1^2>0$, $\beta_2^1>0$, and $\beta_2^2>0$. 

\vskip 0.2cm
\noindent{\it CASE II: $m\in(-\frac 1 2,0)$}
\vskip 0.2cm

For $m \in (-\frac 1 2,0)$, we have $A' > 0$ and $C' > 0$.  This implies that corresponding to $s_{34}=\beta_2$,
there are either two real positive solutions or no real solutions. We want to show that there are two positive real solutions. 
Clearly $-\frac{B''}{2A''}<0$, $r\left(-\frac{B''}{2A''}\right)>0$, $A''<0$ and $C''>0$ for $m \in (-\frac 1 2,0)$, so that $r=0$ has one positive and one negative root. 
If we show that the  positive root is less than $s_{34} = \beta_2$, then it follows that $q_2\left(-\frac{B'}{2A'}\right)<0$, and consequently $q_2=0$ 
has two positive solutions when $s_{34}=\beta_2$ (i.e., $\beta_2^1>0$ and $\beta_2^2>0$).

Subtracting  $\beta_2$ from the positive root of $r=0$ yields
\[
\Delta={\frac {-(-{{ m}}^{3}-7\,{{ m}}^{2}-4\,x{ m}+y{ m}-7\,{ m}-8\,x+y+3)}{4 \left( {
 m}+2 \right)  \left( { m}+1 \right) ^{2}}}
\]
where 
$x=\sqrt{-m(m+2)}$, and $y=\sqrt{(m-1)(m+1)(m^2-4m-9)}$.
Numerically it is easy to see that $\Delta<0$ for $m\in (-\frac 1 2,0)$. However, we proceed more rigorously. 
Taking the numerator of the expression above and squaring the expressions for $x$ and $y$ yields three polynomials.
Computing a lex Gr\"obner basis for these three polynomials, we obtain 
\[
m(5m^3+7m^2+3m+9)(m+2)^2(m+1)^2
\]
as one of the basis elements.  It is trivial to see that this polynomial has no roots in $[-\frac1 2,0]$, 
that $\Delta(0)<0$, and that $\Delta(-\frac 1 2)<0$.
Since $\Delta$ is continuous on [-1/2,0], we deduce that $\Delta<0$ for $m\in(-\frac 1 2,0)$,
as desired.

\vskip 0.2cm
\noindent{\it CASE III: $m\in(-\infty,-2)$}
\vskip 0.2cm

If $s_{34}=\beta_1$ and  $m\in(-\infty,-2)$, then $C'<0$.  Since $A'>0$, it follows that $q_2=0$ has  one positive and one negative 
real solutions corresponding to $\beta_1$. 
If $s_{34}=\beta_2$ and $m\in(-\infty,-2)$, then $C'>0$. Since  $r\left(-\frac{B''}{2A''}\right)<0$ and $A''<0$ it follows that 
 $q_2\left(-\frac{B'}{2A'}\right)<0.$  This together with the fact that  $-\frac{B'}{2A'}>0$, $A'>0$ and  $C'>0$ implies that, 
 corresponding to $\beta_2$, $q_2=0$ has two real positive solutions, (i.e.,  $\beta_2^1>0$, and $\beta_2^2>0$).
 
\vskip 0.2 cm 
\noindent$\bullet$~{\it Solutions of $T_5$ with $s_{34}>0$, $s_{14}>0$, and $s_{13}>0$}
\vskip 0.2 cm 

We now study which solutions among the ones found above have  $s_{13}>0$.
The resultant of the second and third equations of $T_5$ with respect to $s_{12}$ is

\[R=2s_{13}^2+((2m+2)s_{34}-2m-4)s_{13}+(-2-m)s_{34}+m+2.\]
If $s_{13}$ is replaced with $s_{14}$, this polynomial transforms into the one that appears in the third equation of $T_5$.
Hence our discussion regarding $s_{14}$ applies to $s_{13}$. Note that since we saturated with respect to $s_{13}-s_{14}$,
every solution must have $s_{13} \neq s_{14}$. Since $R=0$ has two solutions for each $m$ and $s_{34}$, 
when $s_{14}$ is one such solution, $s_{13}$ must be the other.

\vskip 0.2 cm 
\noindent$\bullet$~{\it Solutions of $T_5$ with $s_{34}>0$, $s_{14}>0$, $s_{13}>0$, and  $s_{12}>0$}
\vskip 0.2 cm 

 \begin{figure}[t]
 \centering
 \includegraphics[width=10 cm,keepaspectratio=true]{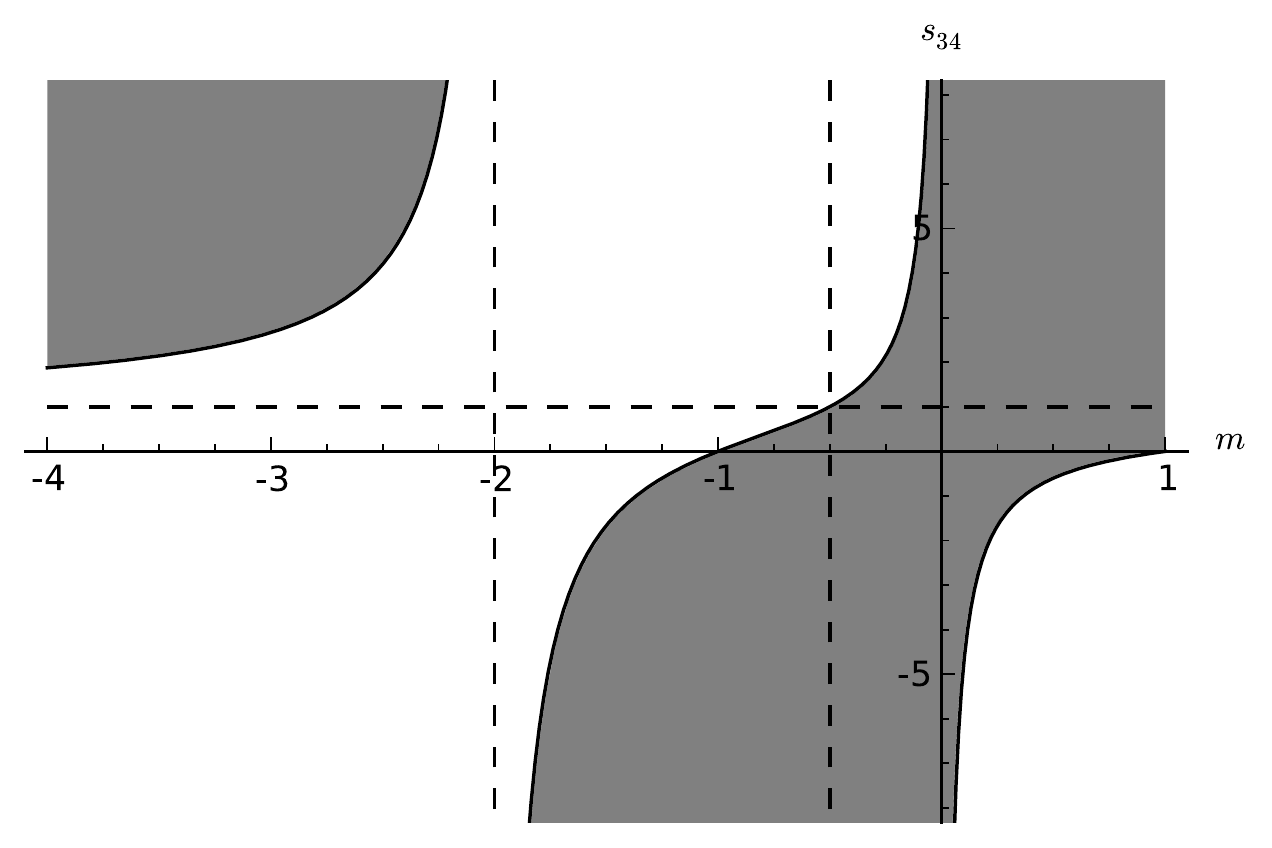}
 \caption{This figure summarizes the analysis of  the sign of $s_{12}$. The grey areas represent regions where $s_{12}>0$.
 The white areas represent regions where $s_{12}<0$. Along the solid lines $s_{12}=0$.
 \label{fig:sign} }
\end{figure}

We now study which solutions among the ones found above have  $s_{12}>0$. Consider the first equation of $T_5$
and rewrite it in the form
\beq\label{eq:s12}
s_{12}=(m^2+2m)s_{34}+1-m^2.
\eeq
This equation has a unique solution for each value of  $s_{34}$ and $m$. Solutions corresponding to $s_{34} = \beta_i$ will be denoted $\alpha_i$. 
We can view this expression as  defining   a function of the two variables $m$ and $s_{34}$. 
The sign of this function, i.e., the sign of $s_{12}$, is summarized in Figure~\ref{fig:sign}. As mentioned earlier, the 
last equation of $T_5$ gives that $s_{34}=\beta_1$ or $s_{34}=\beta_2$, and when $\beta_i > 0$, 
we have $0 < s_{34} < 1$ for $m \in (-\frac 1 2,0) \cup (0,1)$. 
Hence $s_{12} > 0$ for any positive solution of $q_1(s_{34}) = 0$. 

The situation is a bit more delicate when $m \in (-\infty, -2)$. 
In this case, $\beta_2 > 1$, and thus we cannot arrive at a definitive conclusion just looking at Figure~\ref{fig:sign}.
However, using standard methods, one can prove that the following function of~$m$, defined by
\[
d(m)=\beta_2(m)-\frac{m^2-1}{m(m+2)},
\]
is always positive (where the graph of  $\hat{f}(m)=\frac{m^2-1}{m(m+2)}$ is the boundary of the grey region 
in Figure~\ref{fig:sign} in the interval $(-\infty,-2)$).  Thus, $\beta_2$  lies in the region where $s_{12}>0$
when $m\in (-\infty,-2)$.  Consequently, there are always  positive solutions of $T_5$ corresponding to $\beta_2$. 
On the other hand,  $0<\beta_1<\frac 1 2$ for $m \in (-\infty, -2)$, and thus there are no positive solutions of $T_5$ 
corresponding to $\beta_1$ in this interval. 
\vskip 0.4 cm
In summary, we have shown that the solutions $(s_{12},s_{13},s_{14},s_{34})^T$ with all components positive  have the form
\[
\begin{bmatrix}
 \alpha_1\\
 \beta_{1}^2\\
 \beta_1^1\\
 \beta_1
\end{bmatrix}, \quad
\begin{bmatrix}
 \alpha_1\\
 \beta_1^1\\
 \beta_1^2\\
 \beta_1
\end{bmatrix}, \quad
\begin{bmatrix}
 \alpha_2\\
 \beta_2^2\\
 \beta_2^1\\
 \beta_2
\end{bmatrix}, \quad
\begin{bmatrix}
 \alpha_2\\
 \beta_2^1\\
 \beta_2^2\\
 \beta_2
\end{bmatrix},
\]
for $m \in (0,1)$,  while they have the form
\[
\begin{bmatrix}
 \alpha_2\\
 \beta_2^2\\
 \beta_2^1\\
 \beta_2
\end{bmatrix}, \quad
\begin{bmatrix}
 \alpha_2\\
 \beta_2^1\\
 \beta_2^2\\
 \beta_2
\end{bmatrix}
\]
for $m \in (-\infty,-2)\cup(-\frac 1 2,0)$. 

We are left with deciding which solutions are convex and which are concave. For $m\in (0,\infty)$, it follows from 
the main result in~\cite{albouy_symmetry_2008} that, if the kite is convex, it must be a rhombus. Since the rhombus case 
was excluded, these solutions correspond to concave configurations.  For $m \in (-\frac 1 2,0)$, we have that $s_{34}=\beta_2<1$. 
However, according to Proposition~\ref{prop:concave}, since $\lambda'=-1$, the interior side connecting equal vorticities 
in a concave configuration is greater than $1$.  Hence, the configurations are, in this case, convex. 
For $m\in (-\infty, -2)$ we have that $s_{34}=\beta_2>1$. Furthermore, substituting $s_{34}=\beta_2(m)$ into equation~(\ref{eq:s12}),
we obtain an expression that is function of $m$ alone.  Examining this function, it is easy to show that $s_{12}=\alpha_2(m)<1$ for all 
$m\in(-\infty,0)$. Thus, from Propositions \ref{prop:concave} and~\ref{prop:convex}, it follows that the solutions correspond, in this case, to concave configurations. 

This completes the proof of Lemma~\ref{lemma:concavekites}.

\end{proof}

We now consider the other case, where vortices $1$ and~$2$ lie on the axis of symmetry. 
We compute a Gr\"obner basis for the ideal $\langle \mathcal{\bar F}',\mathcal{\bar G}',\bar  e_{CM}',\bar {H}'\rangle$,
where $ \mathcal{\bar F}',\mathcal{\bar G}',\bar  e_{CM}',\bar {\mathcal H}'$ are obtained from 
 $ \mathcal{\widetilde F},\mathcal{\widetilde G},\widetilde e_{CM}$ and $\mathcal{\widetilde H}$, respectively,
by imposing the conditions  $s_{14}=s_{13}$ and $s_{24}=s_{23}$.
We saturate with respect to
$s_{13} - s_{23}$ in order to exclude the rhombus configurations. We also saturate with respect to 
$s_{12}, s_{13}, s_{23}, s_{34}$ and $m$. This computation yields a polynomial system that we will call $G_1$. 
Then we can prove the following lemma.

\begin{lemma}\label{lemma:concavekites1}
 For each value of $m \in (1,\infty)$, system $G_1$ has exactly four real positive solutions corresponding to concave configurations.
 For each value of $m \in (-\infty,-2) \cup (-\frac 1 2,0)$, system $G_1$ has exactly two real positive solutions.
 Such solutions correspond to concave configurations if $m\in(-\frac 1 2,0)$, and to convex configurations if $m\in(-\infty,-2)$.
 For $m=1$, system $G_1$  has exactly two real positive solutions corresponding to an equilateral triangle
 with either vortex 1 or~2 at the center.  There are no positive real solutions for other values of~$m$.
\end{lemma}

\begin{proof}
Computing a full triangular decomposition of $G_1$, with the positivity conditions $s_{12}>0$, $s_{13}>0$, $s_{14}>0$ and $s_{34}>0$,
we decompose $G_1$  into three simpler systems $S_1, S_2$ and $S_3$.  For $m=1$, we have two systems, $S_1$ and $S_2$. 
They are linear, and each of them has a unique solution corresponding to an equilateral triangle configuration with vortex 1 or~2 at the center. 
For $m\in (-\infty,-2)\cup(-\frac 1 2,0)\cup (1,\infty)$, we obtain a third system $S_3$.
This last system, possibly after computing a lexicographic Gr\"obner basis, can be reduced to $T_5$ by performing the 
transformation $(m,s_{12},s_{13},s_{23}, s_{34}) \mapsto (\frac 1 m,s_{34},s_{14},s_{13},s_{12})$ and clearing the denominators. 
This transforms each of the four polynomial equations in $S_3$ to the equations of $T_5$.
Consequently, these cases follow from Lemma~\ref{lemma:concavekites}.
For all other $m$-values the simpler systems have no real solutions.
\end{proof}

\begin{remark}

\begin{enumerate}

\item The proofs of Lemmas \ref{lemma:concavekites} and~\ref{lemma:concavekites1} yield some interesting information on the type of
 kite central configurations.  For example, when $m \in (0,1)$, there are solutions only when the smaller strength vortices
(3 and~4) are on the line of symmetry.  In this case, there are two geometrically distinct configurations for each choice of the central
vortex.

\item  As described in the introduction, the equilateral triangle with a central vortex bifurcates into both concave kite configurations
and asymmetric concave configurations as $m$ decreases through $m = 1$.  The fact that this particular solution shows up in both Sections
\ref{section:AsymmetricCC}  and~\ref{section:concave-kites} is hardly surprising.  In fact, when solving the Albouy-Chenciner equations for
$m=1$ without any restrictions on the variables, this solution occurs four times with a multiplicity of four.  The Hessian matrix
$D^2(H + \lambda I)$ evaluated at this solution has a null space of dimension 2 (excluding the 
``trivial'' eigenvector in the direction of rotation).  One can check that there are eigenvectors in the null space corresponding to
both a symmetric and asymmetric perturbation of the solution.  In fact, the symmetric perturbation
gives rise to the concave kite configurations described above.  The central configuration consisting
of four equal vortices on an equilateral triangle with a vortex at the center is a highly degenerate solution.

\end{enumerate}

\end{remark}

\subsection{Kite configurations: $\lambda' > 0$}
\label{sec:kites-lampos}

In Section~\ref{sec:distancesascoordinates} we observed that, if the vorticities have different
signs, we could have central configurations with $\lambda' > 0$.  In this case, taking
$\Gamma_1 = \Gamma_2 = 1$ and $\Gamma_3 = \Gamma_4 = m < 0$, the configuration is necessarily convex with vortices 1 and~2 on one
diagonal and vortices 3 and~4 on the other.   In this section, we first show that these convex solutions must contain a line of symmetry.
Based on the main result in~\cite{albouy_symmetry_2008} (applicable to our problem only when $m > 0$),
we might expect each solution to contain two lines of symmetry, forming a rhombus.  We show here that this
is not necessarily true.  When $\lambda' > 0$ and $m \in (m^\ast, -1/2)$, where $m^\ast \approx -0.5951$, 
there exists a family of convex kite central configurations that are not rhombi.  There does exist a family of rhombi when $\lambda' > 0$,
as discussed in Section~\ref{rhombus_sec}.
In fact, this family of rhombi undergoes a pitchfork bifurcation at $m = m^\ast$ that results in the convex kites.

\begin{proposition}
Let $\Gamma_1 = \Gamma_2 = 1$ and $\Gamma_3 = \Gamma_4 = m < 0$.  Any strictly planar
solution to the Albouy-Chenciner equations with $\lambda' = 1$ must possess a line of symmetry.
\label{prop:symmetry-lampos}
\end{proposition}

\begin{proof}
The type of computations required are very similar to those used in Sections \ref{section:AsymmetricCC} and~\ref{section:concave-kites}.
Consider the Albouy-Chenciner and the unsymmetrized Albouy-Chenciner equations together with the Cayley-Menger determinant 
and Dziobek equations with the normalization $\lambda'=1$. Take a Gr\"obner basis of the ideal generated by such equations and 
saturate with respect to some of the variables to eliminate possible solutions where one of the mutual distances has zero length. 
Then, saturating with respect to the differences 
$(s_{13}-s_{24}),~(s_{23}-s_{14})$, $(s_{13}-s_{14})$, $(s_{23}-s_{24})$, $(s_{13}-s_{23})$ and~$(s_{14}-s_{24})$,
yields a system of 12 polynomial equations.  Analyzing this system by the method of triangular decomposition of 
semialgebraic systems (with  the inequalities $s_{ij}>0)$ yields an empty triangular system. 
Therefore, due to the saturation and the Symmetry Lemma, there are no asymmetric solutions in this case. 
\end{proof}

To find the solutions with $\lambda' > 0$, we impose the symmetry on the configuration and use equations~(\ref{eqDziobek1}) 
since the areas $A_i$ are easily computable.  Suppose that vortices 1 and~2 are on the axis
of symmetry and set $r_{14} = r_{13}$ and $r_{24} = r_{23}$.  We impose the scaling $r_{34} = 1$ rather than
choose $\lambda' = 1$.   Introduce two new variables $x$ and~$y$ measuring the distance between the intersection of the
two diagonals and vortices 1 and~2, respectively.  Thus, we have the simple geometric equations
\begin{equation}
x^2 + \frac{1}{4} = r_{13}^2, \quad y^2 + \frac{1}{4} = r_{23}^2 \quad  \mbox{and} \quad  x + y = r_{12}.
\label{eq:geom-lampos}
\end{equation}
Equations~(\ref{eq:geom-lampos}) taken together with 
$r_{34} = 1$ imply that the Cayley-Menger determinant $e_{CM}$ vanishes.

Under this setup, convex configurations correspond
to $x > 0$ and $y > 0$, while concave configurations have $xy < 0$.  The oriented areas of the four triangles are given by $A_1 = y/2$, $A_2 = x/2$
and $A_3 = A_4 = -r_{12}/4$, where the signs are taken without loss of generality.
Note that one of the Dziobek equations in~(\ref{eq:Dzio}) is automatically satisfied.  The other equation yields an expression
for $\lambda'$ given by
$$
\lambda' \; = \;  \frac{r_{13}^2 r_{23}^2 - r_{12}^2}{r_{12}^2 (r_{13}^2 + r_{23}^2) - r_{13}^2 r_{23}^2 (r_{12}^2 + 1)} .
$$
Equations~(\ref{eqDziobek1}), the formula for $\lambda'$ and equations~(\ref{eq:geom-lampos}) yield
a polynomial system in the variables $r_{12}, r_{13}, r_{23}, x, y, \lambda'$ and $m$.
We then saturate this system with respect to $r_{13}, r_{23}, r_{13} - r_{23}$ and~$r_{13} + r_{23}$
to eliminate the rhombus solutions and insure the mutual distance variables are nonzero.   Denote the resulting polynomial
system as $P_{{\rm kite}}$.  Analyzing Gr\"obner bases of the ideal generated by $P_{{\rm kite}}$ with different elimination orderings 
yields the following theorem.  Recall that $m^\ast \approx -0.5951$ is the only real root
of the cubic $9m^3 + 3m^2 + 7m + 5$. 

\begin{theorem}
For $m^\ast < m < -1/2$, there exists four convex kite configurations with vortices 1 and~2 on the axis of
symmetry.  These solutions have $\lambda' > 0$ and are not rhombi.   There are no other strictly kite
solutions (with vortices 1 and~2 on the axis of symmetry) with $\lambda' > 0$ and $m < 0$.
\label{thm:kite-lampos}
\end{theorem}

\begin{proof}
Let $G_{{\rm kite}}$ be the Gr\"obner basis of the ideal generated by $P_{{\rm kite}}$ 
with respect to the lex order where $r_{23} > r_{13} > r_{12} > \lambda' > y > x > m$.
The first element in $G_{{\rm kite}}$ is an even 8th-degree polynomial in $x$ with coefficients in $m$.   
Letting $z = x^2$, this polynomial is given by
$$
\begin{array}{cl}
\zeta_m(z) = & 256 m^2 (m+2) (2m+1)^2 \, z^4  - 256 m (9m^4+23m^3+17m^2-m-3) \, z^3 \\[0.05in]
& + (1728m^5 + 3136m^4 + 992m^3 -384m^2 + 64m + 128) \, z^2 \\[0.05in]
& + (-432m^5 -336m^4 + 48m^3 - 80m^2 + 16m + 64) \, z + (m + 2)^3 .
\end{array}
$$
Computing a Gr\"obner basis with the same ordering of variables except that $x > y > m$ 
yields the same polynomial with $z = y^2$.  The discriminant 
of $\zeta_m$ is a positive multiple of
$$
m^4 (m+1)^3 (2m+1) (9m^3+3m^2+7m+5) (9m^2+4m-1)^2 (m-1)^2 (q_k(m))^2   
$$
where $q_k(m) = 2m^5+8m^3+14m^2+4m-1.$  Note that the discriminant vanishes at $m = m^\ast$.
It also vanishes at $m = -1/2$, where $\zeta_m(z)$ becomes a cubic polynomial with roots
$3/4$ and $-1/4$ (multiplicity 2).  It is straight-forward to check that the discriminant
is strictly negative for $m^\ast < m < -1/2$.  Consequently, $\zeta_m(z)$ has precisely two real roots
for each $m$-value in this interval.  

To see that these roots are always positive, we first use resultants to
compute the repeated root of $\zeta_m(z)$ when $m = m^\ast$.  This yields the value $z^\ast \approx 1.9566$,
the only root of the cubic $c_u(z) = 320z^3 - 656z^2 + 60z - 3$.  Next consider the ideal in $\mathbb{C}[z,m]$ 
generated by the system $\{ \zeta_m, c_u \}$, which has $(z^\ast, m^\ast)$ in its variety.  Computing 
a Gr\"{o}bner basis for this ideal which eliminates $z$, we obtain a polynomial in $m$ that has no roots
for $m^\ast < m \leq -1/2$.  Since $\zeta_m(z^\ast) < 0$ when $m = -1/2$, it follows that $\zeta_m(z^\ast) < 0$
for any $m \in (m^\ast, -1/2)$.  Then, since the constant term and leading coefficient of $\zeta_m$ are
both positive for $m^\ast < m < -1/2$, the two real roots of $\zeta_m$ must be positive.

Choose $x$ to be the positive square root of one of the roots of $\zeta_m$.  The second element
in the Gr\"obner basis $G_{{\rm kite}}$ is linear in $y$, and the coefficient of $y$ is non-zero for
$m^\ast < m < -1/2$.    Thus, by the Extension Theorem, we can extend a solution $(x, m)$ of
$\zeta_m(x) = 0$ to a unique partial solution $(y, x, m)$, where $y$ must be real.  By considering 
the other polynomials in $G_{{\rm kite}}$, three of which are given by equations~(\ref{eq:geom-lampos}),
the Extension Theorem repeatedly applies to extend $(y, x, m)$ to a unique full solution of system $P_{{\rm kite}}$
where $x, r_{12}, r_{13}, r_{23}$ are all positive and $r_{13} \neq r_{23}$.  As expected,
many of the elements in $G_{{\rm kite}}$ featuring the saturation variable for $r_{13} - r_{23}$
have a leading coefficient that is a multiple of the key cubic $9m^3+3m^2+7m+5$.
 
To see that $\lambda'$ is positive,
we compute a Gr\"obner basis for $P_{{\rm kite}}$ which eliminates
all variables except $\lambda'$ and $m$.  The first element in this Gr\"obner basis,
a quadratic polynomial in $\lambda'$ with coefficients in $m$, is
\begin{equation}
2 m^2 (m+1) \, \lambda'^{\, 2} + (4m - 1)(m + 1)^2 \, \lambda' + m(m + 2)(2m + 1).
\label{eq:lamprime-m}
\end{equation}
It is straight-forward to show that the roots of this quadratic are real and positive for
$m^\ast < m < -1/2$.  Consequently, $\lambda' > 0$ and the configuration corresponding to
the solution guaranteed by the Extension Theorem is convex.
Therefore, $y > 0$ is also assured, and by symmetry, $y$ must be the positive square root of
the remaining real root of $\zeta_m$.  Since we have saturated to eliminate the rhombus solutions,
we must have $x \neq y$.  Thus, there are two solutions: one where $x > \sqrt{z^\ast} > y > 0$, and one where
$y >  \sqrt{z^\ast} > x > 0$.  The full count of four solutions comes from reflecting each of these solutions about the
axis of symmetry (or simply interchanging vortices 3 and~4).  This proves the first part of the theorem.

To see that there are no other solutions (other than a rhombus) with $\lambda' > 0$, first consider the case
$m \in (-1/2,0)$.  According to equation~(\ref{eq:lamprime-m}), the two possible choices for $\lambda'$ are real, with opposite sign.
Denote the larger, positive root as $\lambda'_+$ and suppose there was a real solution to system 
$P_{{\rm kite}}$ with $\lambda' = \lambda'_+ > 0$. 
Since the quadratic~(\ref{eq:lamprime-m}) is strictly negative when $\lambda' = 3$ and $m \in (-1/2,0)$, we can assume that
$\lambda'_+ > 3$.  Next, the top two equations in~(\ref{eqDziobek1}) imply that 
\begin{equation}
4m^2 x y (1 + \lambda') \; = \;  1 + \lambda' r_{12}^2.
\label{eq:cc-contrad}
\end{equation}
If our real solution has $\lambda' > 0$, then it is convex and $xy > 0$.  Since $\lambda'_+ > 3$, we have that
$$
1 + \lambda'_+ (x^2 + y^2) + xy(\lambda'_+  - 1) \; > \;  0.
$$
Taking $m \in (-1/2,0)$, this in turn implies that
$$
4m^2 xy (1 + \lambda'_+) \; < \; xy(1 + \lambda'_+) \; < \; 1 + \lambda'_+ (x+y)^2 ,
$$
which violates equation~(\ref{eq:cc-contrad}).  Therefore, any real solution for $m \in (-1/2,0)$ must have
$\lambda' < 0$.

Recall that for $m = -1/2$, $\zeta_m$ reduces to a cubic polynomial with only one positive root at $\zeta = 3/4$.
Therefore, the only possible solution to system $P_{{\rm kite}}$ at $m = -1/2$ has $x = y = \sqrt{3}/2$, a rhombus solution.
At $m = m^\ast$, $\zeta_m$ has only one positive root at~$z^\ast$ (multiplicity two), which also implies that
$x = y$.  For the case $-1 < m < m^\ast$, the discriminant of $\zeta_m$ is positive; however, using Lemma~\ref{lemma:quartic-roots}
it is straight-forward to check that all the roots of $\zeta_m$ are complex.  Hence there are no real solutions
to system~$P_{{\rm kite}}$ when $-1 < m < m^\ast$.  

For $m < -1$, the discriminant of $\zeta_m$ is negative so there are two real roots.  To determine their sign,
we note that $\zeta_m(-1/4) = 128 m^3 (m+1) (2m+1)$ is strictly negative for $-2 < m < -1$.  Since both the leading coefficient
and the constant term of $\zeta_m$ are positive when $-2 < m < -1$, it follows that the real roots of $\zeta_m$ are negative
in this case, and there are no real solutions to system~$P_{{\rm kite}}$.  
For $m < -2$, the signs of the leading coefficient and constant term both flip to become negative.
Since $\zeta_m(3/4) = -128 (m-1) (2m+1)^2$ is strictly positive when $m < -2$, it follows that $\zeta_m$ has 
two positive real roots.  However, examining the roots of the quadratic in equation~(\ref{eq:lamprime-m}),
we see that $\lambda' < 0$ whenever $m < -2$.  Finally, when $m=-2$, $\zeta_m$ reduces to
a cubic with only one real root at zero.  Thus, we have shown that for $m < 0$ and $m \not \in (m^\ast, -1/2)$,
there are either no real solutions or only solutions with $\lambda' < 0$.  This completes the proof of
Theorem~\ref{thm:kite-lampos}.
\end{proof}

\begin{remark}

\begin{enumerate}

\item Since Theorem~\ref{thm:kite-lampos} applies for all $m < 0$, we do not need to consider the case
with vortices 3 and~4 on the axis of symmetry.  If there existed a solution for some $m \in (-1,0)$ having
$\lambda' > 0$ and vortices 3 and~4 on the axis of symmetry, we could rescale the vortex strengths by
$1/m$ and relabel the vortices to obtain a solution with $m < -1$,
$\lambda' > 0$, and vortices 1 and~2 on the axis of symmetry.  But this contradicts
the last statement of the theorem.

\item  For completeness, we note that the cases having real, positive solutions with $\lambda' < 0$ agree with
the results in Section~\ref{section:concave-kites}.  For example, when $-1/2 < m < 0$, we find
concave kite configurations with vortices 1 and~2 on the axis of symmetry, as predicted by
Theorem~\ref{thm:concavekites}.  For $m < -2$, we find convex kite solutions with vortices 1 and~2 on
the axis of symmetry.  Rescaling the vortex strengths by $1/m$ and relabeling the vortices gives a
family of convex kites with vortices 3 and~4 on the axis of symmetry and $-1/2 < m < 0$.  This also
concurs with Theorem~\ref{thm:concavekites}.

\end{enumerate}

\end{remark}

\subsection{The rhombus solutions} 
\label{rhombus_sec}

In this final section we study convex kite configurations that have two lines of symmetry, namely rhombus configuations. 
Here we have four congruent exterior sides, $r_{13} = r_{14} = r_{23} = r_{24}$ with the diagonals satisfying
the relation $r_{12}^2 + r_{34}^2 = 4 r_{13}^2$.  The areas satisfy $A_1 = A_2 = -A_3 = -A_4$.  
The first and last equations in~(\ref{eqDziobek2}) are clearly satisfied.  The middle four equations in~(\ref{eqDziobek2}) 
are equivalent.  They, along with the value of $\lambda'$ determined by~(\ref{eq:Dzio}), yield a formula for~$m$ given by
\begin{equation}
m \; = \;  \frac{x^2(3-x^2)}{3x^2 - 1},
\label{eq:RhomMass}
\end{equation}
where $x = r_{34}/r_{12}$ is the ratio between the diagonals of the rhombus.
Note that $m > 0$ if and only if $1/\sqrt{3} < x < \sqrt{3}$.  As usual, we restrict to the case
$-1 < m \leq 1$.


Unlike the isosceles trapezoid case, here we have solutions for $m < 0$.  In fact, there
are two geometrically distinct families of rhombi when $m < 0$.
This can be seen by inverting equation~(\ref{eq:RhomMass}), which yields
\begin{equation}
x^2 \; = \;  \frac{1}{2} \left( 3 - 3m \pm \sqrt{(3-3m)^2 + 4m} \, \right),
\label{eq:rhombus-fams}
\end{equation}
where $+$ is taken if $m > 0$ and $+$ or $-$ may be chosen if $m < 0$.   Using equation~(\ref{eq:lambda}), we compute that
\begin{equation}
\lambda \; = \;  \frac{4(m^2 + 4m + 1)}{r_{12}^2( 2 + 3m - 3m^2 \pm  m \sqrt{9m^2 - 14m + 9} \, )},
\label{eq:lam-rhombus}
\end{equation}
with the same sign choices as for $x^2$.  Note that the numerator of~(\ref{eq:lam-rhombus}) vanishes
at $m = -2 + \sqrt{3} \approx -0.2679$.  The denominator also vanishes at this special value, but only when $+$ is chosen.
For the rhombus solution when $+$ is taken in both (\ref{eq:rhombus-fams}) and~(\ref{eq:lam-rhombus}),
the value of the angular velocity $\lambda$ is always positive (thus $\lambda' < 0$) and monotonically increasing in $m$
for $m \in (-1,1]$.  In contrast, for the rhombus solution when 
$-$ is taken in both (\ref{eq:rhombus-fams}) and~(\ref{eq:lam-rhombus}) ($m < 0$ only), 
we have $\lambda > 0$ (thus $\lambda' < 0$) only for $-2+\sqrt{3} < m < 0$.  At the special value
$m = -2 + \sqrt{3}$, the rhombus relative equilibrium actually becomes an equilibrium
(see Section~\ref{symmetric_equilibria}).  As $m$ decreases through $-2 + \sqrt{3}$,
the direction of rotation flips, $\lambda$ becomes negative and $\lambda'$ becomes positive.
We summarize our conclusions in the following theorem.

\begin{theorem}
There exists two one-parameter families of rhombi relative equilibria with vortex strengths
$\Gamma_1 = \Gamma_2 = 1$ and $\Gamma_3 = \Gamma_4 = m$.  The vortices 1 and~2 lie on opposite
sides of each other, as do vortices 3 and~4.  Let $\beta = 3 - 3m$.
The mutual distances are given by
\begin{equation}
\left( \frac{r_{34}}{r_{12}} \right)^2 \; = \;  \frac{1}{2} \left( \beta \pm \sqrt{\beta^2 + 4m} \, \right) \; \;  \mbox{and} \; \;
\left( \frac{r_{13}}{r_{12}} \right)^2 \; = \;  \frac{1}{8} \left( \beta+2 \pm \sqrt{\beta^2 + 4m} \, \right),
\label{eq:rhom-thm-sols}
\end{equation}
describing two distinct solutions.  Taking $+$ in~(\ref{eq:rhom-thm-sols}) yields a solution
for $m \in (-1,1]$ that always has $\lambda > 0$.  Taking $-$ in~(\ref{eq:rhom-thm-sols}) yields a solution
for $m \in (-1,0)$ that has $\lambda > 0$ for $m \in (-2 + \sqrt{3}, 0)$, but
$\lambda < 0$ for $m \in (-1, -2 + \sqrt{3})$.  At $m = -2 + \sqrt{3}$, the $-$ solution
becomes an equilibrium.  The case $m=1$ reduces to the square.  For $m > 0$, 
the larger pair of vortices lie on the shorter diagonal.
\end{theorem}

\begin{proof}
The formula for $r_{13}/r_{12}$ comes from $1+x^2 = 4(r_{13}/r_{12})^2$.  For the case $m > 0$, it is easy to show
from equation~(\ref{eq:RhomMass}) that $m < 1$ if and only if $1 < x < \sqrt{3}$.  This verifies the last statement in the
theorem.
\end{proof}

 \begin{figure}[tb]
 \centering
 \includegraphics[width=7.5 cm,keepaspectratio=true]{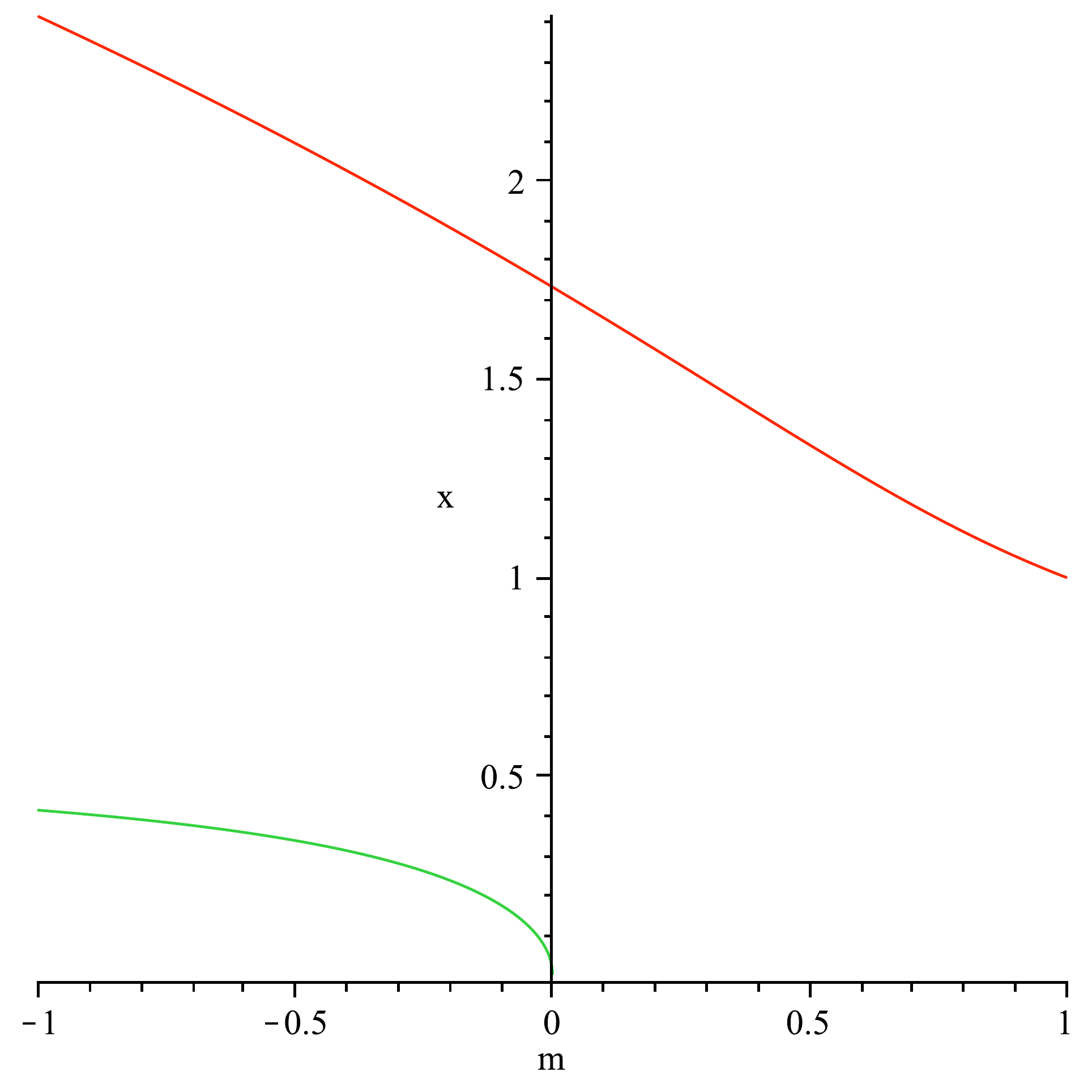}
 \caption{A plot of $x$ versus $m$ for the two rhombus families, where $x = r_{34}/r_{12}$ is the ratio between the diagonals of the rhombus.
 \label{fig:rhombus-xm} }
\end{figure}

\begin{remark}

\begin{enumerate}

\item   One way to visualize the two rhombus solutions is captured in Figure~\ref{fig:rhombus-xm}, where a plot
of $x = r_{34}/r_{12}$ versus $m$ is shown.  Beginning with the square at $m=1$, as $m$ decreases,
the ratio of the diagonals of the rhombus increases from 1 to $\sqrt{3}$, with the larger pair of
vortices lying on the shorter diagonal.  At $m=0$ a bifurcation occurs, and a new family is born emerging
out of a binary collision between vortices 3 and~4.  Unlike the previous family, which continues past the bifurcation,
this family has the larger pair of vortices (in absolute value) on the longest diagonal (see Figure~\ref{fig:rhombi_m=-0.3}).  
As $m \rightarrow -1^+$, $x \rightarrow \sqrt{2} - 1$ for the new family while $x \rightarrow \sqrt{2} + 1$ for the previous family.

\item   The rhombus solution when $-$ is chosen in equation~(\ref{eq:rhom-thm-sols}) undergoes a pitchfork bifurcation at
the special parameter value $m^\ast \approx -0.5951$.  As $m$ increases through $m^\ast$, the rhombus solution bifurcates into
two convex kite solutions with the positive strength vortices on the axis of symmetry.  The two kites are distinguished by
whether $r_{13} > r_{23}$ or $r_{13} < r_{23}$.  Since the rhombus solution continues to exist past the bifurcation, we have a
pitchfork bifurcation.  The Hessian matrix $D^2(H + \lambda I)$ evaluated at the $-$ rhombus solution
at $m = m^\ast$ has a null space of dimension 1 (excluding the ``trivial'' eigenvector in the direction of rotation)
and contains an eigenvector corresponding to a perturbation in the direction of the convex kite solution
found in Section~\ref{sec:kites-lampos}.

\item  The fact that there are two geometrically distinct rhombus solutions for $m < 0$ (plus two convex kites when
$m^\ast < m < -1/2$) indicates that there is {\em not} a unique (up to symmetry) convex central configuration
for these four choices of vorticities.  This contrasts with the Newtonian four-body problem where it is thought
(although unproven) that there is a unique convex central configuration for any ordering of four positive masses.

\end{enumerate}

\end{remark}

 \begin{figure}[tb]
 \centering
 \includegraphics[height=6.5 cm,keepaspectratio=true]{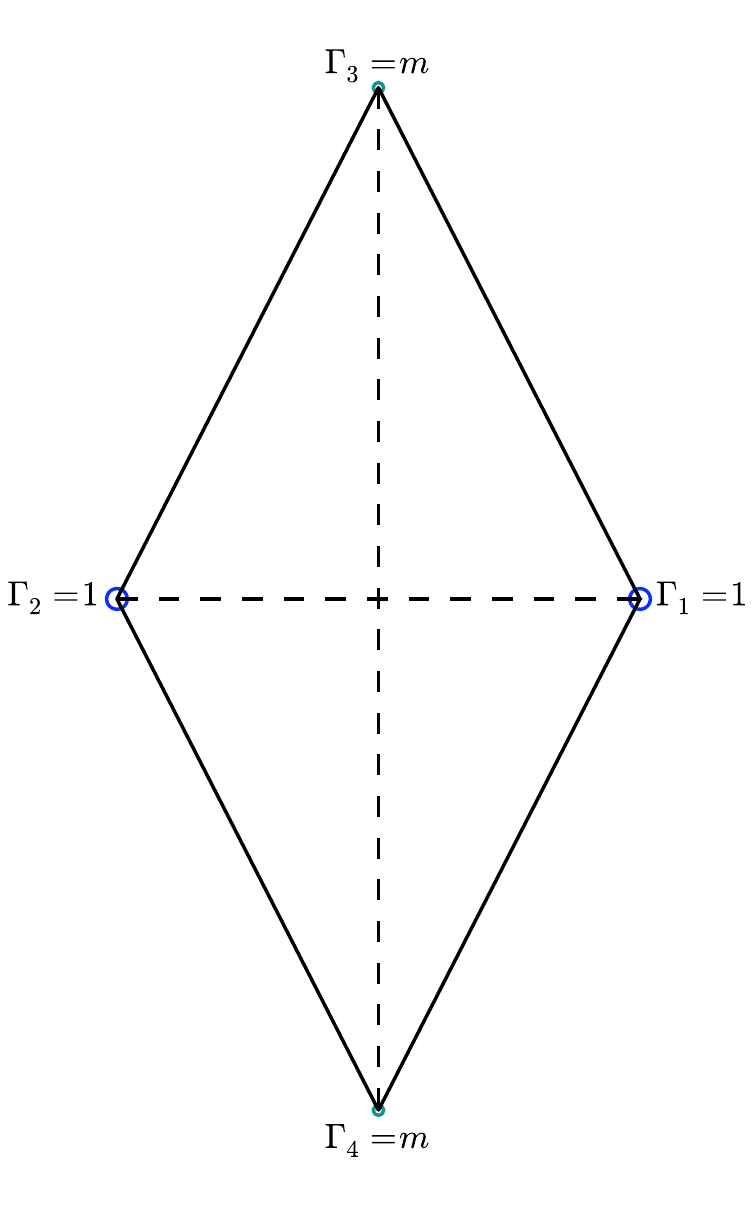}
 \hspace{0.4in}
 \includegraphics[width=6.5 cm,keepaspectratio=true]{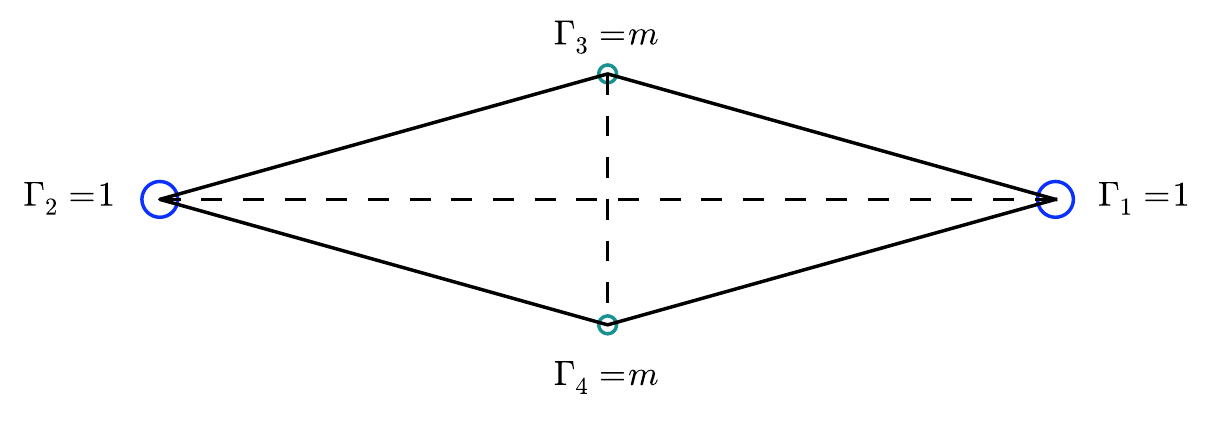}
 \caption{The two distinct rhombi relative equilibria when $m = -0.3$.  The solutions rotate in opposite
 directions.}
  \label{fig:rhombi_m=-0.3} 
\end{figure}


\vs

\nin {\bf Acknowledgements:}
Part of this work was carried out when the authors were visiting the American Institute of Mathematics
in May of 2011.  We gratefully acknowledge their hospitality and support.  GR was also supported by a grant
from the National Science Foundation and MS was supported by a NSERC Discovery Grant.

\pagebreak

\bibliography{Papers,Books,My_Papers}{}
\bibliographystyle{amsplain}

\end{document}